\documentclass[notitlepage,11pt,reqno]{amsart}
\usepackage{amssymb,amsmath,tocvsec2}
\usepackage[mathscr]{eucal}
\usepackage[breaklinks=true]{hyperref}

\usepackage{pgf,tikz}
\usetikzlibrary{arrows}
\usepackage[letterpaper]{geometry}
\newcommand{\ba}{\ensuremath{\mathbf a}}
\newcommand{\bn}{\ensuremath{\mathbf n}}
\newcommand{\bp}{\ensuremath{\mathbf p}}
\newcommand{\bq}{\ensuremath{\mathbf q}}
\newcommand{\bu}{\ensuremath{\mathbf u}}
\newcommand{\bv}{\ensuremath{\mathbf v}}

\newcommand{\R}{\ensuremath{\mathbb R}}
\newcommand{\K}{\ensuremath{\mathbb K}}
\newcommand{\D}{\ensuremath{\mathcal{D}}}

\renewcommand{\subset}{\subseteq}
\renewcommand{\geq}{\geqslant}
\renewcommand{\leq}{\leqslant}

\DeclareMathOperator{\Id}{Id}
\DeclareMathOperator{\diag}{diag}

\newtheorem{theorem}{Theorem}[section]
\newtheorem{lemma}[theorem]{Lemma}
\newtheorem{corollary}[theorem]{Corollary}
\newtheorem{proposition}[theorem]{Proposition}

\newtheorem{utheorem}{\textrm{\textbf{Theorem}}}

\theoremstyle{definition}
\newtheorem{definition}[theorem]{Definition}
\newtheorem{remark}[theorem]{Remark}
\newtheorem{example}[theorem]{Example}

\numberwithin{equation}{section}

\begin{document}
\vspace*{-4mm}

\title[Blowup-polynomial of a metric space; stable polynomials and
graphs]{The blowup-polynomial of a metric space:\\ connections to stable
polynomials, graphs and their distance spectra}

\author{Projesh Nath Choudhury}
\address[P.N.~Choudhury]{Department of Mathematics, Indian Institute of
Technology Gandhinagar, Palaj, Gandhinagar 382055, India}
\email{\tt projeshnc@iitgn.ac.in}

\author{Apoorva Khare}
\address[A.~Khare]{Department of Mathematics, Indian Institute of Science;
Analysis and Probability Research Group; Bangalore 560012, India}
\email{\tt khare@iisc.ac.in}



\begin{abstract}
To every finite metric space $X$, including all connected unweighted
graphs with the minimum edge-distance metric, we attach an invariant that
we call its blowup-polynomial $p_X(\{ n_x : x \in X \})$. This is
obtained from the blowup $X[{\bf n}]$ -- which contains $n_x$ copies of
each point $x$ -- by computing the determinant of the distance matrix of
$X[{\bf n}]$ and removing an exponential factor. We prove that as a
function of the sizes $n_x$, $p_X({\bf n})$ is a polynomial, is
multi-affine, and is real-stable. This naturally associates a hitherto
unstudied delta-matroid to each metric space $X$; we produce another
novel delta-matroid for each tree, which interestingly does not
generalize to all graphs.

We next specialize to the case of $X = G$ a connected unweighted graph --
so $p_G$ is ``partially symmetric'' in $\{ n_v : v \in V(G) \}$ -- and
show three further results:
(a)~We show that the polynomial $p_G$ is indeed a graph invariant, in
that $p_G$ and its symmetries recover the graph $G$ and its isometries,
respectively.
(b)~We show that the univariate specialization $u_G(x) := p_G(x,\dots,x)$
is a transform of the characteristic polynomial of the distance matrix
$D_G$; this connects the blowup-polynomial of $G$ to the well-studied
``distance spectrum'' of $G$.
(c)~We obtain a novel characterization of complete multipartite graphs,
as precisely those for which the ``homogenization at $-1$'' of $p_G({\bf
n})$ is real-stable (equivalently, Lorentzian, or strongly/completely
log-concave), if and only if the normalization of $p_G(-{\bf n})$ is
strongly Rayleigh.
\end{abstract}

\date{\today}
\maketitle

\vspace*{-2mm}

\tableofcontents
\settocdepth{section}

\section{The blowup-polynomial of a metric space and its distance matrix}

This work aims to provide novel connections between metric geometry, the
geometry of (real) polynomials, and algebraic combinatorics via partially
symmetric functions. In particular, we introduce and study a polynomial
graph-invariant for each graph, which to our knowledge is novel.

\subsection{Motivations}

The original motivation for our paper came from the study of distance
matrices $D_G$ of graphs $G$ -- on both the algebraic and spectral sides:
\begin{itemize}
\item On the algebraic side, Graham and Pollak \cite{Graham-Pollak}
initiated the study of $D_G$ by proving: if $T_k$ is a tree on $k$ nodes,
then $\det D_{T_k}$ is independent of the tree structure and depends only
on $k$. By now, many variants of such results are proved, for trees as
well as several other families of graphs, including with $q$-analogues,
weightings, and combinations of both of these. (See e.g.\ \cite{CK-tree}
and its references for a list of such papers, results, and their common
unification.)

\item Following the above work~\cite{Graham-Pollak}, Graham also
worked on the spectral side, and with Lov\'asz, studied in~\cite{GL} the
distance matrix of a tree, including computing its inverse and
characteristic polynomial. This has since led to the intensive study of
the roots, i.e.\ the ``distance spectrum'', for trees and other graphs.
See e.g.\ the survey~\cite{AH} for more on distance spectra.
\end{itemize}

A well-studied problem in spectral graph theory involves understanding
which graphs are \textit{distance co-spectral} -- i.e., for which graphs
$H' \not\cong K'$, if any, do $D_{H'}, D_{K'}$ have the same spectra.
Many such examples exist; see e.g.\ the references in~\cite{DL}. In
particular, the characteristic polynomial of $D_G$ does not ``detect''
the graph $G$. It is thus natural to seek some other byproduct of $D_G$
which does -- i.e., which recovers $G$ up to isometry. In this paper, we
find such a (to our knowledge) novel graph invariant: a multivariate
polynomial, which we call the \textit{blowup-polynomial of $G$}, and
which does detect $G$. Remarkably, this polynomial turns out to have
several additional attractive properties:
\begin{itemize}
\item It is multi-affine in its arguments.
\item It is also real-stable, so that its ``support'' yields a hitherto
unexplored delta-matroid.
\item The blowup-polynomial simultaneously encodes the determinants of
\textit{all} graph-blowups of $G$ (defined presently), thereby connecting
with the algebraic side (see the next paragraph).
\item Its ``univariate specialization'' is a transformation of the
characteristic polynomial of $D_G$, thereby connecting with the spectral
side as well.
\end{itemize}
Thus, the blowup-polynomial that we introduce, connects distance spectra
for graphs -- and more generally, for finite metric spaces -- to other
well-studied objects, including real-stable/Lorentzian polynomials and
delta-matroids.

On the algebraic side, a natural question involves asking if there are
graph families $\{ G_i : i \in I \}$ (like trees on $k$ vertices) for
which the scalars $\det (D_{G_i})$ behave ``nicely'' as a function of $i
\in I$. As stated above, the family of blowups of a fixed graph $G$
(which help answer the preceding ``spectral'' question) not only answer
this question positively as well, but the nature of the answer --
multi-affine polynomiality -- is desirable in conjunction with its
real-stability. In fact, we will obtain many of these results, both
spectral and algebraic, in greater generality: for arbitrary finite
metric spaces.

The key construction required for all of these contributions is that of a
blowup, and we begin by defining it more generally, for arbitrary
metric spaces that are discrete (i.e., every point is isolated).

\begin{definition}\label{Ddef}
Given a metric space $(X,d)$ with all isolated points, and a function
$\bn : X \to \mathbb{Z}_{>0}$, the \emph{$\bn$-blowup} of $X$ is the
metric space $X[\bn]$ obtained by creating $n_x := \bn(x)$ copies of each
point $x$ (also termed blowups of $x$).
Define the distance between copies of distinct points $x \neq y$ in $X$
to still be $d(x,y)$, and between distinct copies of the same point to be
$2 d(x,X \setminus \{x\}) = 2 \inf_{y \in X \setminus \{ x \}} d(x,y)$.

Also define the \emph{distance matrix} $D_X$ and the \emph{modified
distance matrix} $\D_X$ of $X$ via:
\begin{equation}
D_X := (d(x,y))_{x,y \in X}, \qquad
\D_X := D_X + \diag(2 d(x, X \setminus \{ x \}))_{x \in X}.
\end{equation}
\end{definition}

Notice for completeness that the above construction applied to a
non-discrete metric space does not yield a metric; and that blowups of
$X$ are ``compatible'' with isometries of $X$ (see~\eqref{Eisom}).
We also remark that this notion of blowup seems to be relatively less
studied in the literature, and differs from several other variants in the
literature -- for metric spaces e.g.~\cite{metric} or for graphs
e.g.~\cite{Liu}. However, the variant studied in this paper was
previously studied for the special case of unweighted graphs, see
e.g.~\cite{HHN,KKOT,KSS} in extremal and probabilistic graph theory.

\subsection{Defining the blowup-polynomial; Euclidean embeddings}

We now describe some of the results in this work, beginning with metric
embeddings. Recall that the complete information about a (finite) metric
space is encoded into its distance matrix $D_X$ (or equivalently, in the
off-diagonal part of $\D_X$). Metric spaces are useful in many
sub-disciplines of the mathematical sciences, and have been studied for
over a century. For instance, a well-studied question in metric geometry
involves understanding metric embeddings. In 1910, Fr\'echet
showed~\cite{Frechet0} that every finite metric space with $k+1$ points
isometrically embeds into $\R^k$ with the supnorm. Similarly, a
celebrated 1935 theorem of Schoenberg \cite{Schoenberg35} (following
Menger's works~\cite{Menger0,Menger}) says:

\begin{theorem}[Schoenberg, \cite{Schoenberg35}]\label{Tschoenberg}
A finite metric space $X = \{ x_0, \dots, x_k \}$ isometrically embeds
inside Euclidean space $(\R^r, \| \cdot \|_2)$ if and only if its
modified Cayley--Menger matrix
\begin{equation}
(d(x_0,x_i)^2 + d(x_0,x_j)^2 - d(x_i,x_j)^2)_{i,j=1}^k
\end{equation}
is positive semidefinite, with rank at most $r$.
\end{theorem}

As an aside, the determinant of this matrix is related to the volume of a
polytope with vertices $x_i$ (beginning with classical work of
Cayley~\cite{Cayley}), and the Cayley--Menger matrix itself connects to
the principle of trilateration/triangulation that underlies the GPS
system.

Returning to the present work, our goal is to study the distance matrix
of a finite metric space vis-a-vis its blowups. We begin with a
``negative'' result from metric geometry. Note that every blowup of a
finite metric space embeds into $\R^k$ (for some $k$) equipped with the
supnorm, by Fr\'echet's aforementioned result. In contrast, we employ
Schoenberg's theorem~\ref{Tschoenberg} to show that the same is far from
true when considering the Euclidean metric. Namely, given a finite metric
space $X$, we characterize all blowups $X[\bn]$ that embed in some
Euclidean space $(\R^k, \| \cdot \|_2)$. Since $X$ embeds into $X[\bn]$,
a necessary condition is that $X$ itself should be Euclidean. With this
in mind, we have:

\begin{utheorem}\label{Teuclidean}
Suppose $X = \{ x_1, \dots, x_k \}$ is a finite metric subspace of
Euclidean space $(\R^r, \| \cdot \|_2)$. Given positive integers $\{
n_{x_i} : 1 \leq i \leq k \}$, not all of which equal $1$, the following
are equivalent:
\begin{enumerate}
\item The blowup $X[\bn]$ isometrically embeds into some Euclidean space
$(\R^{r'}, \| \cdot \|_2)$.

\item Either $k=1$ and $\bn$ is arbitrary (then by convention $X[\bn]$ is
a simplex); or $k>1$ and there exists a unique $1 \leq j \leq k$ such
that $n_{x_j} = 2$. In this case, we moreover have:
(a)~$n_{x_i} = 1\ \forall i \neq j$,
(b)~$x_j$ is not in the affine hull/span $V$ of $\{ x_i : i \neq j \}$,
and
(c)~the unique point $v \in V$ closest to $x_j$, lies in $X$.
\end{enumerate}
If these conditions hold, one can take $r' = r$ and $X[\bn] = X \sqcup \{
2v - x_j \}$.
\end{utheorem}

Given the preceding result, we turn away from metric geometry, and
instead focus on studying the family of blowups $X[\bn]$ -- through their
distance matrices $D_{X[\bn]}$ (which contain all of the information on
$X[\bn]$). Drawing inspiration from Graham and Pollak
\cite{Graham-Pollak}, we focus on one of the simplest invariants of this
family of matrices: their determinants, and the (possibly algebraic)
nature of the dependence of $\det D_{X[\bn]}$ on $\bn$. In this paper, we
show that the function $: \bn \mapsto \det D_{X[\bn]}$ possesses several
attractive properties. First, $\det D_{X[\bn]}$ is a polynomial function
in the \textit{sizes} $n_x$ of the blowup, up to an exponential factor:

\begin{utheorem}\label{Tmetricmatrix}
Given $(X,d)$ a finite metric space, and a tuple of positive integers
$\bn := (n_x)_{x \in X} \in \mathbb{Z}_{>0}^X$, the function $\bn \mapsto
\det D_{X[\bn]}$ is a multi-affine polynomial $p_X(\bn)$ in the $n_x$
(i.e., its monomials are squarefree in the $n_x$), times the exponential
function
\[
\prod_{x \in X} (-2 \; d(x, X \setminus \{ x \}))^{n_x-1}.
\]

Moreover, the polynomial $p_X(\bn)$ has constant term
$p_X({\bf 0}) = \prod_{x \in X} (-2 \; d(x, X \setminus \{ x \}))$, and
linear term $-p_X({\bf 0}) \sum_{x \in X} n_x$.
\end{utheorem}

Theorem~\ref{Tmetricmatrix} follows from a stronger one proved below. See
Theorem~\ref{Tmonoid}, which shows in particular that not only do the
conclusions of Theorem~\ref{Tmetricmatrix} hold over an arbitrary
commutative ring, but moreover, the blowup-polynomial $p_X(\bn)$ is a
polynomial function in the variables $\bn = \{ n_x : x \in X \}$ as well
as the entries of the ``original'' distance matrix $D_X$ -- and moreover,
it is squarefree/multi-affine in all of these arguments (where we treat
all entries of $D_X$ to be ``independent'' variables).

We also refine the final assertions of Theorem~\ref{Tmetricmatrix}, by
isolating in Proposition~\ref{Pcoeff} the coefficient of \textit{every}
monomial in $p_X(\bn)$. That proposition moreover provides a sufficient
condition under which the coefficients of two monomials in $p_X(\bn)$ are
equal.

Theorem~\ref{Tmetricmatrix} leads us to introduce the following notion,
for an arbitrary finite metric space (e.g., every finite, connected,
$\mathbb{R}_{>0}$-weighted graph).

\begin{definition}\label{Dblowup}
Define the \emph{(multivariate) blowup-polynomial} of a finite metric
space $(X,d)$ to be $p_X(\bn)$, where the $n_x$ are thought of as
indeterminates. We write out a closed-form expression in the proof of
Theorem~\ref{Tmetricmatrix} -- see Equation~\eqref{Emetricpoly}.

In this paper, we also study a specialization of this polynomial. Define
the \emph{univariate blowup-polynomial} of $(X,d)$ to be $u_X(n) :=
p_X(n,n,\dots,n)$, where $n$ is thought of as an indeterminate.
\end{definition}

\begin{remark}\label{Rzariski}
Definition~\ref{Dblowup} requires a small clarification.
The polynomial map (by Theorem~\ref{Tmetricmatrix}) \[ \bn \mapsto \det
D_{X[\bn]} \cdot \prod_{x \in X} (-2 d(x, X \setminus \{ x \}))^{1 -
n_x}, \qquad \bn \in \mathbb{Z}_{>0}^k \] can be extended from the
Zariski dense subset $\mathbb{Z}_{>0}^k$ to all of $\R^k$. (Zariski
density is explained during the proof of Theorem~\ref{Tmetricmatrix}
below.) Since $\R$ is an infinite field, this polynomial map on $\R^k$
may now be identified with a polynomial, which is precisely $p_X(-)$, a
polynomial in $|X|$ variables (which we will denote by $\{ n_x : x \in X
\}$ throughout the paper, via a mild abuse of notation). Now setting all
arguments to be the same indeterminate yields the univariate
blowup-polynomial of $(X,d)$.
\end{remark}

\subsection{Real-stability}

We next discuss the blowup-polynomial $p_X(\cdot)$ and its univariate
specialization $u_X(\cdot)$ from the viewpoint of root-location
properties. As we will see, the polynomial $u_X(n) = p_X(n,n,\dots,n)$
always turns out to be real-rooted in $n$. In fact, even more is true.
Recall that in recent times, the notion of real-rootedness has been
studied in a much more powerful avatar: real-stability. Our next result
strengthens the real-rootedness of $u_X(\cdot)$ to the second attractive
property of $p_X(\cdot)$ -- namely, real-stability:

\begin{utheorem}\label{Tstable}
The blowup-polynomial $p_X(\bn)$ of every finite metric space $(X,d)$ is
real-stable in $\{ n_x \}$. (Hence its univariate specialization $u_X(n)
= p_X(n,n,\dots,n)$ is always real-rooted.)
\end{utheorem}

Recall that real-stable polynomials are simply ones with real
coefficients, which do not vanish when all arguments are constrained to
lie in the (open) upper half-plane $\Im(z) > 0$. Such polynomials have
been intensively studied in recent years, with a vast number of
applications. For instance, they were famously used in celebrated works
of Borcea--Br\"and\'en (e.g.~\cite{BB1,BB2,BB3}) and
Marcus--Spielman--Srivastava~\cite{MSS1,MSS2} to prove longstanding
conjectures (including of Kadison--Singer, Johnson, Bilu--Linial,
Lubotzky, and others), construct expander graphs, and vastly extend the
Laguerre--P\'olya--Schur program \cite{Laguerre1,Polya1913,Polya-Schur}
from the turn of the 20th century (among other applications).

Theorem~\ref{Tstable} reveals that for all finite metric spaces -- in
particular, for all finite connected graphs -- the blowup-polynomial is
indeed multi-affine and real-stable. The class of multi-affine
real-stable polynomials has been characterized in \cite[Theorem
5.6]{Branden} and \cite[Theorem 3]{WW}. (For a connection to matroids,
see \cite{Branden}, \cite{COSW}.) To the best of our knowledge,
blowup-polynomials $p_X(\bn)$ provide novel examples/realizations of
multi-affine real-stable polynomials.

\subsection{Graph metric spaces: symmetries, complete multipartite
graphs}

We now turn from the metric-geometric Theorem~\ref{Teuclidean}, the
algebraic Theorem~\ref{Tmetricmatrix}, and the analysis-themed
Theorem~\ref{Tstable}, to a more combinatorial theme, by restricting from
metric spaces to graphs. Here we present two ``main theorems'' and one
proposition.

\subsubsection{Graph invariants and symmetries}

Having shown that $\det D_{X[\bn]}$ is a polynomial in $\bn$ (times an
exponential factor), and that $p_X(\cdot)$ is always real-stable, our
next result explains a third attractive property of $p_X(\cdot)$:
\textit{The blowup-polynomial of a graph $X = G$ is indeed a (novel)
graph invariant}. To formally state this result, we begin by re-examining
the blowup-construction for graphs and their distance matrices.

A distinguished sub-class of discrete metric spaces is that of finite
simple connected unweighted graphs $G$ (so, without parallel/multiple
edges or self-loops). Here the distance between two nodes $v,w$ is
defined to be the (edge-)length of any shortest path joining $v,w$. In
this paper, we term such objects \textbf{graph metric spaces}. Note that
the blowup $G[\bn]$ is \textit{a priori} only defined as a metric space;
we now adjust the definition to make it a graph.

\begin{definition}
Given a graph metric space $G = (V,E)$, and a tuple $\bn = (n_v : v \in
V)$, the \emph{$\bn$-blowup} of $G$ is defined to be the graph $G[\bn]$
-- with $n_v$ copies of each vertex $v$ -- such that a copy of $v$ and
one of $w$ are adjacent in $G[\bn]$ if and only if $v \neq w$ are
adjacent in $G$.
\end{definition}

\noindent (For example, the $\bn$-blowup of a complete graph is a
complete multipartite graph.) Now note that if $G$ is a graph metric
space, then so is $G[\bn]$ for all tuples $\bn \in
\mathbb{Z}_{>0}^{|V|}$. The results stated above thus apply to every such
graph $G$ -- more precisely, to the distance matrices of the blowups of
$G$.

To motivate our next result, now specifically for graph metric spaces, we
first relate the symmetries of the graph with those of its
blowup-polynomial $p_G(\bn)$. Suppose a graph metric space $G = (V,E)$
has a structural (i.e., adjacency-preserving) symmetry $\Psi : V \to V$
-- i.e., an \textit{(auto-)isometry} as a metric space. Denoting the
corresponding relabeled graph metric space by $\Psi(G)$,
\begin{equation}\label{Eisom} D_G = D_{\Psi(G)}, \qquad \D_G =
\D_{\Psi(G)}, \qquad p_G(\bn) \equiv p_{\Psi(G)}(\bn). \end{equation}

It is thus natural to ask if the converse holds -- i.e., if $p_G(\cdot)$
helps recover the group of auto-isometries of $G$. A stronger result
would be if $p_G$ recovers $G$ itself (up to isometry). We show that both
of these hold:

\begin{utheorem}\label{Tisom}
Given a graph metric space $G = (V,E)$ and a bijection $\Psi : V \to V$,
the symmetries of the polynomial $p_G$ equal the isometries of $G$. In
particular, any (equivalently all) of the statements in~\eqref{Eisom}
hold, if and only if $\Psi$ is an isometry of $G$. More strongly, the
polynomial $p_G(\bn)$ recovers the graph metric space $G$ (up to
isometry). However, this does not hold for the polynomial $u_G$.
\end{utheorem}

As the proof reveals, one in fact needs only the homogeneous quadratic
part of $p_G$, i.e.\ its Hessian matrix $((\partial_{n_v}
\partial_{n_{v'}} p_G)({\bf 0}_V))_{v,v' \in V}$, to recover the graph
and its isometries. Moreover, this associates to every graph a
\textit{partially symmetric polynomial}, whose symmetries are precisely
the graph-isometries.
\medskip

Our next result works more generally in metric spaces $X$, hence is
stated over them. Note that the polynomial $p_X(\bn)$ is ``partially
symmetric'', depending on the symmetries (or isometries) of the distance
matrix (or metric space). Indeed, partial symmetry is as much as one can
hope for, because it turns out that ``full'' symmetry (in all variables
$n_x$) occurs precisely in one situation:

\begin{proposition}\label{Tsymm}
Given a finite metric space $X$, the following are equivalent:
\begin{enumerate}
\item The polynomial $p_X(\bn)$ is symmetric in the variables $\{ n_x, \
x \in X \}$.
\item The metric $d_X$ is a rescaled discrete metric:
$d_X(x,y) = c {\bf 1}_{x \neq y}\ \forall x,y \in X$, for some $c>0$.
\end{enumerate}
\end{proposition}

\subsubsection{Complete multipartite graphs: novel characterization via
stability}

The remainder of this section returns back to graphs. We next present an
interesting byproduct of the above results: a novel characterization of
the class of complete multipartite graphs. Begin by observing from the
proof of Theorem~\ref{Tstable} that the polynomials $p_G(\cdot)$ are
stable because of a determinantal representation (followed by inversion).
However, they do not enjoy two related properties:
\begin{enumerate}
\item $p_G(\cdot)$ is not homogeneous.

\item The coefficients of the multi-affine polynomial $p_G(\cdot)$ are
not all of the same sign; in particular, they cannot form a probability
distribution on the subsets of $\{ 1, \dots, k \}$ (corresponding to the
various monomials in $p_G(\cdot)$). In fact, even the constant and linear
terms have opposite signs, by the final assertion in
Theorem~\ref{Tmetricmatrix}.
\end{enumerate}

These two (unavailable) properties of real-stable polynomials are indeed
important and well-studied in the literature. Corresponding to the
preceding numbering: \begin{enumerate} \item Very recently, Br\"and\'en
and Huh~\cite{BH} introduced and studied a distinguished class of
homogeneous real polynomials, which they termed \textit{Lorentzian}
polynomials (defined below). Relatedly, Gurvits \cite{Gurvits} /
Anari--Oveis Gharan--Vinzant \cite{AGV} defined strongly/completely
log-concave polynomials, also defined below. These classes of polynomials
have several interesting properties as well as applications; see
e.g.~\cite{AGV3,AGV,BH,Gurvits}, and related/follow-up works.

\item Recall that strongly Rayleigh measures are probability measures on
the power set of $\{ 1, \dots, k \}$ whose generating (multi-affine)
polynomials are real-stable. These were introduced and studied by Borcea,
Br\"and\'en, and Liggett in the fundamental work~\cite{BBL}. This work
developed the theory of negative association/dependence for such
measures, and enabled the authors to prove several conjectures of
Liggett, Pemantle, and Wagner, among other achievements. \end{enumerate}

Given that $p_G(\cdot)$ is always real-stable, a natural question is if
one can characterize those graphs for which a certain homogenization of
$p_G(\cdot)$ is Lorentzian, or a suitable normalization is strongly
Rayleigh. The standard mathematical way to address obstacle~(1) above is
to ``projectivize'' using a new variable $z_0$, while for obstacle~(2) we
evaluate at $(-z_1, \dots, -z_k)$, where we use $z_j$ instead of
$n_{x_j}$ to denote complex variables. Thus, our next result proceeds via
homogenization at $-z_0$:

\begin{utheorem}\label{Tlorentz}
Say $G = (V,E)$ with $|V|=k$. Define the
{\em homogenized blowup-polynomial}
\begin{equation}
\widetilde{p}_G(z_0, z_1, \dots, z_k) := (-z_0)^k p_G \left(
\frac{z_1}{-z_0}, \dots, \frac{z_k}{-z_0} \right) \in \mathbb{R}[z_0,
z_1, \dots, z_k].
\end{equation}
Then the following are equivalent:
\begin{enumerate}
\item The polynomial $\widetilde{p}_G(z_0, z_1, \dots, z_k)$ is
real-stable.

\item The polynomial $\widetilde{p}_G(\cdot)$ has all coefficients (of
the monomials $z_0^{k - |J|} \prod_{j \in J} z_j$) non-negative.

\item We have $(-1)^k p_G(-1,\dots,-1) > 0$, and the normalized
``reflected'' polynomial
\[
(z_1, \dots, z_k) \quad \mapsto \quad \frac{p_G(-z_1, \dots,
-z_k)}{p_G(-1,\dots,-1)}
\]
is {\em strongly Rayleigh}. In other words, this (multi-affine)
polynomial is real-stable and has non-negative coefficients (of all
monomials $\prod_{j \in J} z_j$), which sum up to $1$.

\item The modified distance matrix $\D_G$ (see Definition~\ref{Ddef}) is
positive semidefinite.

\item $G$ is a complete multipartite graph.
\end{enumerate}
\end{utheorem}

Theorem~\ref{Tlorentz} is a novel characterization result of complete
multipartite graphs in the literature, in terms of real stability and the
strong(ly) Rayleigh property.
Moreover, given the remarks preceding Theorem~\ref{Tlorentz}, we present
three further equivalences to these characterizations:

\begin{corollary}\label{Clorentz}
Definitions in Section~\ref{Slorentz}. The assertions in
Theorem~\ref{Tlorentz} are further equivalent to:
\begin{enumerate}
\setcounter{enumi}{5}
\item The polynomial $\widetilde{p}_G(z_0, \dots, z_k)$ is Lorentzian.

\item The polynomial $\widetilde{p}_G(z_0, \dots, z_k)$ is strongly
log-concave.

\item The polynomial $\widetilde{p}_G(z_0, \dots, z_k)$ is completely
log-concave.
\end{enumerate} 
\end{corollary}

We quickly explain the corollary. Theorem~\ref{Tlorentz}(1) implies
$\widetilde{p}_G$ is Lorentzian (see \cite{BH,COSW}), which implies
Theorem~\ref{Tlorentz}(2). The other equivalences follow from
\cite[Theorem 2.30]{BH}, which shows that -- for any real homogeneous
polynomial -- assertions~(7), (8) here are equivalent to
$\widetilde{p}_G$ being Lorentzian.

\begin{remark}
As we see in the proof of Theorem~\ref{Tlorentz}, when $\D_G$ is positive
semidefinite, the homogeneous polynomial $\widetilde{p}_G(z_0, \dots,
z_k)$ has a determinantal representation, i.e.,
\[
\widetilde{p}_G(z_0, \dots, z_k) = c \cdot \det( z_0 \Id_k + \sum_{j=1}^k
z_j A_j),
\]
with all $A_j$ positive semidefinite and $c \in \R$. In
Proposition~\ref{Pmixed} below, we further compute the \textit{mixed
characteristic polynomial} of these matrices $A_j$
(see~\eqref{Emixeddefn} for the definition), and show that up to a
scalar, it equals the ``inversion'' of the univariate blowup-polynomial,
i.e.\ $z_0^k u_G(z_0^{-1})$.
\end{remark}

\begin{remark}
We also show that the univariate polynomial $u_G(x)$ is intimately
related to the characteristic polynomial of $D_G$ (i.e, the ``distance
spectrum'' of $G$), whose study was one of our original motivations. See
Proposition~\ref{Pdistancespectrum} and the subsequent discussion, for
precise statements.
\end{remark}

\subsection{Two novel delta-matroids}

We conclude with a related byproduct: two novel constructions of
delta-matroids, one for every finite metric space and the other for each
tree graph. Recall that a \textit{delta-matroid} consists of a finite
``ground set'' $E$ and a nonempty collection of \textit{feasible} subsets
$\mathcal{F} \subset 2^E$, satisfying $\bigcup_{F \in \mathcal{F}} F = E$
as well as the symmetric exchange axiom: \textit{Given $A,B \in
\mathcal{F}$ and $x \in A \Delta B$ (their symmetric difference), there
exists $y \in A \Delta B$ such that $A \Delta \{ x, y \} \in
\mathcal{F}$.} Delta-matroids were introduced by Bouchet
in~\cite{Bouchet1} as a generalization of the notion of matroids.

Each (skew-)symmetric matrix $A_{k \times k}$ over a field yields a
\textit{linear delta-matroid} $\mathcal{M}_A$ as follows. Given any
matrix $A_{k \times k}$, let $E := \{ 1, \dots, k \}$ and let a subset $F
\subset E$ belong to $\mathcal{M}_A$ if either $F$ is empty or the
principal submatrix $A_{F \times F}$ is nonsingular. In~\cite{Bouchet2},
Bouchet showed that if $A$ is (skew-)symmetric, then the set system
$\mathcal{M}_A$ is indeed a delta-matroid, which is said to be
\textit{linear}.

We now return to the blowup-polynomial. First recall a 2007 result of
Br\"and\'en \cite{Branden}: given a multi-affine real-stable polynomial,
the set of monomials with nonzero coefficients forms a delta-matroid.
Thus, from $p_X(\bn)$ we obtain a delta-matroid, which as we will explain
is linear:

\begin{corollary}\label{Cdeltamatroid}
Given a finite metric space $(X,d)$, the set of monomials with nonzero
coefficients in $p_X(\bn)$ forms the linear delta-matroid
$\mathcal{M}_{\D_X}$.
\end{corollary}

\begin{definition}
We term $\mathcal{M}_{\D_X}$ the \emph{blowup delta-matroid} of $(X,d)$. 
\end{definition}

The blowup delta-matroid $\mathcal{M}_{\D_X}$ is -- even for $X$ a finite
connected unweighted graph -- a novel construction that arises out of
metric geometry rather than combinatorics, and one that seems to be
unexplored in the literature (and unknown to experts).
Of course, it is a simple, direct consequence of Br\"and\'en's result in
\cite{Branden}.
However, the next delta-matroid is less direct to show:

\begin{utheorem}\label{Ttree-blowup}
Suppose $T = (V,E)$ is a finite connected unweighted tree with $|V| \geq
2$. Define the set system $\mathcal{M}'(T)$ to comprise all subsets
$I \subset V$, except for the ones that contain two vertices $v_1 \neq
v_2$ in $I$ such that the Steiner tree $T(I)$ has $v_1, v_2$ as leaves
with a common neighbor. Then $\mathcal{M}'(T)$ is a delta-matroid, which
does not equal $\mathcal{M}_{D_T}$ for every path graph $T = P_k$, $k
\geq 9$.
\end{utheorem}

We further prove, this notion of (in)feasible subsets in
$\mathcal{M}'(T)$ does not generalize to all graphs. Thus,
$\mathcal{M}'(T)$ is a \textit{combinatorial} (not matrix-theoretic)
delta-matroid that is also unstudied in the literature to our knowledge,
and which arises from every tree, but interestingly, not from all
graphs.

As a closing statement here: in addition to further exploring the
real-stable polynomials $p_G(\bn)$, it would be interesting to obtain
connections between these delta-matroids $\mathcal{M}_{\D_G}$ and
$\mathcal{M}'(T)$, and others known in the literature from combinatorics,
polynomial geometry, and algebra.

\subsection*{Organization of the paper}

The remainder of the paper is devoted to proving the above
Theorems~\ref{Teuclidean} through~\ref{Ttree-blowup}; this will require
developing several preliminaries along the way. The paper is clustered by
theme; thus, the next two sections and the final one respectively
involve, primarily:
\begin{itemize}
\item (commutative) algebraic methods -- to prove the polynomiality of
$p_X(\cdot)$ (Theorem \ref{Tmetricmatrix}), and to
characterize those $X$ for which it is a symmetric polynomial
(Proposition \ref{Tsymm});

\item methods from real-stability and analysis -- to show $p_X(\cdot)$ is
real-stable (Theorem~\ref{Tstable});

\item metric geometry -- to characterize for a given Euclidean finite
metric space $X$, all blowups that remain Euclidean
(Theorem~\ref{Teuclidean}), and to write down a related ``tropical''
version of Schoenberg's Euclidean embedding theorem
from~\cite{Schoenberg35}.
\end{itemize}

In the remaining Section~\ref{Sgraphs}, we prove
Theorems~\ref{Tisom}--\ref{Ttree-blowup}. In greater detail: we focus on
the special case of $X = G$ a finite simple connected unweighted graph,
with the minimum edge-distance metric. After equating the isometries of
$G$ with the symmetries of $p_G(\bn)$, and recovering $G$ from
$p_G(\bn)$, we prove the aforementioned characterization of complete
multipartite graphs $G$ in terms of $\widetilde{p}_G$ being real-stable,
or $p_G(-\bn) / p_G(-1, \dots, -1)$ being strongly Rayleigh. Next, we
discuss a family of blowup-polynomials from this viewpoint of ``partial''
symmetry. We also connect $u_G(x)$ to the characteristic polynomial of
$D_G$, hence to the distance spectrum of $G$. Finally, we introduce the
delta-matroid $\mathcal{M}'(T)$ for every tree, and explore its relation
to the blowup delta-matroid $\mathcal{M}_{\D_T}$ (for $T$ a path), as
well as extensions to general graphs. We end with two Appendices that
contain supplementary details and results.

We conclude this section on a philosophical note. Our approach in this
work adheres to the maxim that the multivariate polynomial is a natural,
general, and more powerful object than its univariate specialization.
This is of course famously manifested in the recent explosion of activity
in the geometry of polynomials, via the study of real-stable polynomials
by Borcea--Br\"and\'en and other researchers; but also shows up in
several other settings -- we refer the reader to the survey~\cite{Sokal2}
by Sokal for additional instances. (E.g., a specific occurrence is in the
extreme simplicity of the proof of the multivariate Brown--Colbourn
conjecture~\cite{Royle-Sokal,Sokal1}, as opposed to the involved proof in
the univariate case~\cite{Wagner}.)

\section{Algebraic results: the blowup-polynomial and its full
symmetry}\label{S2}

We begin this section by proving Theorem~\ref{Tmetricmatrix} in ``full''
algebraic (and greater mathematical) generality, over an arbitrary unital
commutative ring $R$. We require the following notation.

\begin{definition}
Fix positive integers $k, n_1, \dots, n_k > 0$, and vectors $\bp_i, \bq_i
\in R^{n_i}$ for all $1 \leq i \leq k$.
\begin{enumerate}
\item For these parameters, define the \emph{blowup-monoid} to be the
collection $\mathcal{M}_\bn(R) := R^k \times R^{k \times k}$. We write a
typical element as a pair $(\ba, D)$, where in coordinates, $\ba =
(a_i)^T$ and $D = (d_{ij})$.

\item Given $(\ba, D) \in \mathcal{M}_\bn(R)$, define $M(\ba,D)$ to be
the square matrix of dimension $n_1 + \cdots + n_k$ with $k^2$ blocks,
whose $(i,j)$-block for $1 \leq i,j \leq k$ is $\delta_{i,j} a_i
\Id_{n_i} + d_{ij} \bp_i \bq_j^T$. Also define $\Delta_\ba \in R^{k
\times k}$ to be the diagonal matrix with $(i,i)$ entry $a_i$, and
\[
N(\ba,D) := \Delta_\ba + \diag(\bq_1^T \bp_1, \dots, \bq_k^T \bp_k) \cdot
D \ \in R^{k \times k}.
\]

\item Given $\ba, \ba' \in R^k$, define $\ba \circ \ba' := (a_1 a'_1,
\dots, a_k a'_k)^T \in R^k$.
\end{enumerate}
\end{definition}

The set $\mathcal{M}_\bn(R)$ is of course a group under addition, but we
are interested in the following non-standard monoid structure on it.

\begin{lemma}\label{Lmonoid}
The set $\mathcal{M}_\bn(R)$ is a monoid under the product
\begin{align*}
(\ba,D) \circ (\ba',D') := &\ (\ba \circ \ba', \Delta_\ba D' + D
\Delta_{\ba'} + D \cdot \diag(\bq_1^T \bp_1, \dots, \bq_k^T \bp_k) \cdot
D')\\
= &\ (\ba \circ \ba', \Delta_\ba D' + D N(\ba',D')),
\end{align*}
and with identity element $((1,\dots,1)^T, 0_{k \times k})$.
\end{lemma}

With this notation in place, we now present the ``general'' formulation
of Theorem~\ref{Tmetricmatrix}:

\begin{theorem}\label{Tmonoid}
Fix integers $k, n_1, \dots, n_k$ and vectors $\bp_i, \bq_i$ as above.
Let $K := n_1 + \cdots + n_k$.
\begin{enumerate}
\item The following map is a morphism of monoids:
\[
\Psi : (\mathcal{M}_\bn(R), \circ) \to (R^{K \times K}, 
\cdot), \qquad (\ba,D) \mapsto M(\ba,D).
\]

\item The determinant of $M(\ba,D)$ equals $\prod_i a_i^{n_i - 1}$ times
a multi-affine polynomial in $a_i, d_{ij}$, and the entries $\bq_i^T
\bp_i$. More precisely,
\begin{equation}\label{Emonoid}
\det M(\ba,D) = \det N(\ba,D) \prod_{i=1}^k a_i^{n_i - 1}.
\end{equation}

\item If all $a_i \in R^\times$ and $N(\ba,D)$ is invertible, then so is
$M(\ba,D)$, and
\[
M(\ba,D)^{-1} = M((a_1^{-1}, \dots, a_k^{-1})^T, -\Delta_\ba^{-1} D
N(\ba,D)^{-1}).
\]
\end{enumerate}
\end{theorem}

Instead of using $N(\ba,D)$ which involves ``post-multiplication'' by
$D$, one can also use $N(\ba,D^T)^T$ in the above results, to obtain
similar formulas that we leave to the interested reader.

\begin{proof}
The first assertion is easy, and it implies the third assertion via
showing that $M(\ba,D)^{-1} M(\ba,D) = \Id_K$. (We show these
computations for completeness in an Appendix.) Thus, it remains to prove
the second assertion.
To proceed, we employ \textit{Zariski density}, as was done in
e.g.~our previous work~\cite{CK-tree}. Namely, we begin by working over
the field of rational functions in $k + k^2 + 2K$ variables
\[
\mathbb{F} := \mathbb{Q}(A_1, \dots, A_k, D_{ij}, Q_i^{(l)}, P_i^{(l)}),
\]
where $A_i, D_{ij}$ (with a slight abuse of notation), and $Q_i^{(l)},
P_i^{(l)}$ -- with $1 \leq i,j \leq k$ and $1 \leq l \leq n_i$ -- serve
as proxies for $a_i, d_{ij}$, and the coordinates of $\bq_i, \bp_i$
respectively. Over this field, we work with
\[
{\bf A} = (A_1, \dots, A_k)^T, \qquad
{\bf Q}_i = (Q_i^{(1)}, \dots, Q_i^{(n_i)})^T, \qquad
{\bf P}_i = (P_i^{(1)}, \dots, P_i^{(n_i)})^T,
\]
and the matrix ${\bf D} = (D_{ij})$; note that ${\bf D}$ has full rank
$r=k$, since $\det {\bf D}$ is a nonzero polynomial over $\mathbb{Q}$,
hence is a unit in $\mathbb{F}$.

Let ${\bf D} = \sum_{j=1}^r \bu_j \bv_j^T$ be any rank-one decomposition.
For each $1 \leq j \leq r$, write $\bu_j = (u_{j1}, \dots, u_{jk})^T$,
and similarly for $\bv_j$. Then $D_{ij} = \sum_{s=1}^r u_{si} v_{sj}$ for
all $i,j$. Now a Schur complement argument (with respect to the $(2,2)$
block below) yields:
\[
\det M({\bf A}, {\bf D}) = \det \begin{pmatrix}
A_1 \Id_{n_1} & \cdots & 0 & \vline & u_{11} {\bf P}_1 & \cdots & u_{r1}
{\bf P}_1 \\
\vdots & \ddots & \vdots & \vline & \vdots & \ddots & \vdots\\
0 & \cdots & A_k \Id_{n_k} & \vline & u_{1k} {\bf P}_k & \cdots & u_{rk}
{\bf P}_k\\
\hline
-v_{11} {\bf Q}_1^T & \cdots & -v_{1k} {\bf Q}_k^T & \vline & & & \\
\vdots & \ddots & \vdots & \vline & & \Id_r & \\
-v_{r1} {\bf Q}_1^T & \cdots & -v_{rk} {\bf Q}_k^T & \vline & & &
\end{pmatrix}.
\]

We next compute the determinant on the right alternately: by using the
Schur complement with respect to the $(1,1)$ block instead. This yields:
\[
\det M({\bf A}, {\bf D}) = \det ( \Id_r + M ) \prod_{i=1}^k A_i^{n_i},
\]
where $M_{r \times r}$ has $(i,j)$ entry
$\sum_{l=1}^k v_{il} \; (A_l^{-1} {\bf Q}_l^T {\bf P}_l) \; u_{jl}$.
But $\det (\Id_r + M)$ is also the determinant of
\[
M' :=
\begin{pmatrix}
& & & \vline & (A_1^{-1} {\bf Q}_1^T {\bf P}_1) u_{11} & \cdots &
(A_1^{-1} {\bf Q}_1^T {\bf P}_1) u_{r1} \\
& \Id_k & & \vline & \vdots & \ddots & \vdots\\
& & & \vline & (A_k^{-1} {\bf Q}_k^T {\bf P}_k) u_{1k} & \cdots &
(A_k^{-1} {\bf Q}_k^T {\bf P}_k) u_{rk}\\
\hline
-v_{11} & \cdots & -v_{1k} & \vline & & & \\
\vdots & \ddots & \vdots & \vline & & \Id_r & \\
-v_{r1} & \cdots & -v_{rk} & \vline & & &
\end{pmatrix},
\]
by taking the Schur complement with respect to its $(1,1)$ block.
Finally, take the Schur complement with respect to the $(2,2)$ block of
$M'$, to obtain:
\begin{align*}
\det M({\bf A}, {\bf D}) = &\ \det M' \prod_{i=1}^k A_i^{n_i} = \det
\left(\Id_k + \Delta_{\bf A}^{-1} \diag( {\bf Q}_1^T {\bf P}_1, \dots,
{\bf Q}_k^T {\bf P}_k) {\bf D} \right) \prod_{i=1}^k A_i^{n_i}\\
= &\ \det N({\bf A}, {\bf D}) \prod_{i=1}^k A_i^{n_i - 1},
\end{align*}

\noindent and this is indeed $\prod_i A_i^{n_i - 1}$ times a multi-affine
polynomial in the claimed variables.

The above reasoning proves the assertion~\eqref{Emonoid} over the field
\[
\mathbb{F} = \mathbb{Q}(A_1, \dots, A_k, D_{ij}, Q_i^{(l)}, P_i^{(l)})
\]
defined above. We now explain how Zariski density helps
prove~\eqref{Emonoid} over every unital commutative ring -- with the key
being that both sides of~\eqref{Emonoid} are \textit{polynomials} in the
variables. Begin by observing that~\eqref{Emonoid} actually holds over
the polynomial (sub)ring
\[
R_0 := \mathbb{Q}[A_1, \dots, A_k, D_{ij}, Q_i^{(l)}, P_i^{(l)}],
\]
but the above proof used the invertibility of the polynomials $A_1,
\dots, A_k, \det (D_{ij})_{i,j=1}^k$.

Now use that $\mathbb{Q}$ is an infinite field; thus, the following
result applies:

\begin{proposition}\label{Pzariski}
The following are equivalent for a field $\mathbb{F}$.
\begin{enumerate}
\item The polynomial ring $\mathbb{F}[x_1, \dots, x_n]$ (for some $n \geq
1$) equals the ring of polynomial functions from affine $n$-space
$\mathbb{A}_{\mathbb{F}}^n \cong \mathbb{F}^n$ to $\mathbb{F}$.

\item The preceding statement holds for every $n \geq 1$.

\item $\mathbb{F}$ is infinite.
\end{enumerate}

Moreover, the nonzero-locus $\mathcal{L}$ of any nonzero polynomial in
$\mathbb{F}[x_1, \dots, x_n]$ with $\mathbb{F}$ an infinite field, is
Zariski dense in $\mathbb{A}_{\mathbb{F}}^n$. In other words, if a
polynomial in $n$ variables equals zero on $\mathcal{L}$, then it
vanishes on all of $\mathbb{A}_{\mathbb{F}}^n \cong \mathbb{F}^n$.
\end{proposition}

\begin{proof}[Proof-sketch]
Clearly $(2) \implies (1)$; and that the contrapositive of $(1) \implies
(3)$ holds follows from the fact that over a finite field $\mathbb{F}_q$,
the nonzero polynomial $x_1^q - x_1$ equals the zero \textit{function}.
The proof of $(1) \implies (3)$ is by induction on $n \geq 1$, and is
left to the reader (or see e.g.\ standard textbooks, or even
\cite{CK-tree}) -- as is the proof of the final assertion.
\end{proof}

By the equivalence in Proposition~\ref{Pzariski}, the above polynomial
ring $R_0$ equals the ring of polynomial functions in the same number of
variables, so~\eqref{Emonoid} now holds over the ring of polynomial
functions in the above $k + k^2 + 2K$ variables -- but only on the
nonzero-locus of the polynomial $(\det {\bf D}) \prod_i A_i$, since we
used $A_i^{-1}$ and the invertibility of ${\bf D}$ in the above proof.

Now for the final touch: as $(\det {\bf D}) \prod_i A_i$ is a nonzero
polynomial, its nonzero-locus is Zariski dense in affine space
$\mathbb{A}_{\mathbb{Q}}^{k + k^2 + 2K}$ (by Proposition~\ref{Pzariski}).
Since the difference of the polynomials in~\eqref{Emonoid} (this is where
we use that $\det(\cdot)$ is a polynomial!) vanishes on the above
nonzero-locus, it does so for all values of $A_i$ and the other
variables. Therefore~\eqref{Emonoid} holds in the ring $R'_0$ of
polynomial \textit{functions} with coefficients in $\mathbb{Q}$, hence
upon restricting to the polynomial subring of $R'_0$ with integer (not
just rational) coefficients -- since the polynomials on both sides
of~\eqref{Emonoid} have integer coefficients. Finally, the proof is
completed by specializing the variables $A_i$ to specific scalars $a_i$
in an arbitrary unital commutative ring $R$, and similarly for the other
variables.
\end{proof}

Theorem~\ref{Tmonoid}, when specialized to $p_i^{(l)} = q_i^{(l)} = 1$
for all $1 \leq i \leq k$ and $1 \leq l \leq n_i$, reveals how to convert
the sizes $n_{x_i}$ in the blowup-matrix $D_{X[\bn]}$ into entries of the
related matrix $N(\ba,D)$. This helps prove a result in the introduction
-- that $\det D_{X[\bn]}$ is a polynomial in $\bn$:

\begin{proof}[Proof of Theorem~\ref{Tmetricmatrix}]
Everything but the final sentence follows from Theorem~\ref{Tmonoid},
specialized to
\begin{align*}
R = \mathbb{R}, \qquad n_i = n_{x_i}, \qquad
d_{ij} = &\ d(x_i, x_j)\ \forall i \neq j, \qquad
d_{ii} = 2 d(x_i, X \setminus \{ x_i \}) = -a_i, \\
D = \D_X = &\ (d_{ij})_{i,j=1}^k, \qquad
p_i^{(l)} = q_i^{(l)} = 1\ \forall 1 \leq l \leq n_i.
\end{align*}
(A word of caution: $d_{ii} \neq d(x_i, x_i)$, and hence $\D_X \neq D_X$:
they differ by a diagonal matrix.)

In particular, $p_X(\bn)$ is a multi-affine polynomial in $\bq_i^T \bp_i
= n_i$. We also write out the blowup-polynomial, useful here and below:
\begin{align}\label{Emetricpoly}
p_X(\bn) = \det &\ N(\ba_X,\D_X), \quad \text{where} \quad \ba_X =
(D_X - \D_X) (1,1,\dots,1)^T,\\
\text{and so} \quad &\ N(\ba_X,\D_X) = \diag((n_{x_i} - 1) 2 d(x_i, X
\setminus \{ x_i \}))_i + (n_{x_i} d(x_i, x_j))_{i,j=1}^k. \notag
\end{align}
Now the constant term is obtained by evaluating $\det N({\bf a}_X, 0_{k
\times k})$, which is easy since $N({\bf a}_X, 0_{k \times k})$ is
diagonal. Similarly, the coefficient of $n_{x_i}$ is obtained by setting
all other $n_{x_{i'}} = 0$ in $\det N(\ba_X,\D_X)$. Expand along the
$i$th column to compute this determinant; now adding these determinants
over all $i$ yields the claimed formula for the linear term.
\end{proof}

As a further refinement of Theorem~\ref{Tmetricmatrix}, we isolate
\textit{every} term in the multi-affine polynomial $p_X(\bn)$.
Two consequences follow:
(a)~a formula relating the blowup-polynomials for a metric space $X$ and
its subspace $Y$; and
(b)~a sufficient condition for two monomials in $p_X(\bn)$ to have equal
coefficients. In order to state and prove these latter two results, we
require the following notion.

\begin{definition}
We say that a metric subspace $Y$ of a finite metric space $(X,d)$ is
\emph{admissible} if for every $y \in Y$, there exists $y' \in Y$ such
that $d(y, X \setminus \{ y \}) = d(y,y')$.
\end{definition}

For example, in every finite simple connected unweighted graph $G$ with
the minimum edge-distance as its metric, a subset $Y$ of vertices is
admissible if and only if the induced subgraph in $G$ on $Y$ has no
isolated vertices.

\begin{proposition}\label{Pcoeff}
Notation as above.
\begin{enumerate}
\item Given any subset $I \subset \{ 1, \dots, k \}$, the coefficient in
$p_X(\bn)$ of $\prod_{i \in I} n_{x_i}$ is
\[
\det (\D_X)_{I \times I} \prod_{j \not\in I} (-2 d(x_j, X \setminus \{
x_j \})) = \det (\D_X)_{I \times I} \prod_{j \not\in I} (-d_{jj}),
\]
with $(\D_X)_{I \times I}$ the principal submatrix of $\D_X$ formed by
the rows and columns indexed by~$I$.

\item Suppose $I \subset \{ 1, \dots, k \}$, and $Y = \{ x_i : i \in I
\}$ is an admissible subspace of $X$. Then,
\[
p_Y(\{ n_{x_i} : i \in I \}) = p_X(\bn)|_{n_{x_j} = 0\; \forall j \not\in
I} \cdot \prod_{j \not\in I} (-2 d(x_j, X \setminus \{ x_j \}))^{-1}.
\]
In particular, if a particular monomial $\prod_{i \in I_0} n_{x_i}$ does
not occur in $p_Y(\cdot)$ for some $I_0 \subset I$, then it does not
occur in $p_X(\cdot)$ either.

\item Suppose two admissible subspaces of $X$, consisting of points
$(y_1, \dots, y_l)$ and $(z_1, \dots, z_l)$, are isometric (here, $1 \leq
l \leq k$). If moreover
\begin{equation}\label{Ehypothesis}
\prod_{i=1}^l d(y_i, X \setminus \{ y_i \}) =
\prod_{i=1}^l d(z_i, X \setminus \{ z_i \}),
\end{equation}
then the coefficients in $p_X(\bn)$ of $\prod_{i=1}^l
n_{y_i}$ and $\prod_{i=1}^l n_{z_i}$ are equal.
\end{enumerate}
\end{proposition}

The final assertion strengthens the (obvious) observation that if $\Psi :
X \to X$ is an isometry, then $p_X(\cdot) \equiv p_{\Psi(X)}(\cdot)$ --
in other words, the polynomial $p_X(\cdot)$ is invariant under the action
of the permutation of the variables $( n_x : x \in X )$ induced by
$\Psi$. This final assertion applies to blowup-polynomials of unweighted
graphs with ``locally homeomorphic neighborhoods'', e.g.~to interior
points and intervals in path graphs (or more generally, banded graphs).
See the opening discussion in Section~\ref{Rgraphsymm}, as well as
Proposition~\ref{Pzeroterms}.

\begin{proof}\hfill
\begin{enumerate}
\item It suffices to compute the coefficient of $\prod_{i \in I} n_{x_i}$
in $p_X(\bn) = \det N(\ba_X,\D_X)$, where $a_i = -2 d(x_i, X \setminus \{
x_i \})\ \forall 1 \leq i \leq k$, and we set all $n_{x_j},\ j \not\in I$
to zero. To evaluate this determinant, notice that for $j \not\in I$, the
$j$th row contains only one nonzero entry, along the main diagonal. Thus,
expand the determinant along the $j$th row for every $j \not\in I$; this
yields $\prod_{j \not\in I} (-d_{jj})$ times the principal minor
$N(\ba_X,\D_X)_{I \times I}$. Moreover, the coefficient of $\prod_{i \in
I} n_{x_i}$ in the expansion of $\det N(\ba_X,\D_X)_{I \times I}$ is the
same as that in expanding $\det N({\bf 0}, \D_X)_{I \times I}$, and this
is precisely $\det (\D_X)_{I \times I}$.

\item Let us use $\ba_X, \D_X$ and $\ba_Y, \D_Y$ for the appropriate data
generated from $X$ and $Y$ respectively. Then the admissibility of $Y$
indicates that $(\ba_X)_I = \ba_Y$ and $(\D_X)_{I \times I} = \D_Y$.
Now a direct computation reveals:
\[
p_X(\bn)|_{n_{x_j} = 0\; \forall j \not\in I} = \det(\Delta_{\ba_Y} +
\Delta_{\bn_Y} \D_Y) \prod_{j \not\in I} (-d_{jj}).
\]
This shows the claimed equation, and the final assertion is an immediate
consequence of it.

\item Let $I', I'' \subset \{ 1, \dots, k \}$ index the points $(y_1,
\dots, y_l)$ and $(z_1, \dots, z_l)$, respectively. Similarly, let $\D_Y,
\D_Z$ denote the respective $l \times l$ matrices (e.g.\ with
off-diagonal entries $d(y_i, y_j)$ and $d(z_i, z_j)$ respectively). The
admissibility of the given subspaces implies that $(\D_X)_{I' \times I'}
= \D_Y$ and $(\D_X)_{I'' \times I''} = \D_Z$. Now use the isometry
between the $y_i$ and $z_i$ (up to relabeling) to deduce that $\det \D_Y
= \det \D_Z$. Via the first part above, it remains to prove that
\[
\prod_{j \not\in I'} (-2 d(x_j, X \setminus \{ x_j \})) =
\prod_{j \not\in I''} (-2 d(x_j, X \setminus \{ x_j \})).
\]
But this indeed holds, since multiplying the left- and right- hand sides
of this equation by the corresponding sides of~\eqref{Ehypothesis} yields
$2^{-l} \prod_{x \in X} (-2d(x, X \setminus \{ x \}))$ on both sides
(once again using admissibility). \qedhere
\end{enumerate}
\end{proof}

We provide some applications of Proposition~\ref{Pcoeff} in later
sections; for now, we apply it to prove that the blowup delta-matroid of
$X$ is linear:

\begin{proof}[Proof of Corollary~\ref{Cdeltamatroid}]
It is immediate from Proposition~\ref{Pcoeff}(1) that the blowup
delta-matroid of $X$ is precisely the linear delta-matroid
$\mathcal{M}_{\D_X}$ (see the paragraph preceding
Corollary~\ref{Cdeltamatroid}).
\end{proof}

We conclude this section by showing another result in the introduction,
which studies when $p_X(\bn)$ is symmetric in the variables $n_x$.

\begin{proof}[Proof of Proposition~\ref{Tsymm}]
First suppose $d_X$ is the discrete metric times a constant $c > 0$. Then
all $a_i = -2c = d_{ii}$. Hence,
\[
\D_X = c {\bf 1}_{k \times k} + c \Id_k \quad \implies \quad
N(\ba_X,\D_X) = -2c \Id_k + \diag(n_{x_1}, \dots, n_{x_k}) \D_X
\]
and this is a rank-one update of the diagonal matrix $\mathbf{\Delta} := c
\diag(n_{x_1}, \dots, n_{x_k}) -2c \Id_k$. Hence
\begin{align}\label{Ecomplete}
p_X(\bn) = \det N(\ba_X,\D_X) = &\ \det (\mathbf{\Delta}) + c \cdot
(1,\dots,1) {\rm adj} (\mathbf{\Delta}) (n_{x_1}, \dots,
n_{x_k})^T\notag\\
= &\ c^k \left( \prod_{i=1}^k (n_{x_i} - 2) + \sum_{i=1}^k n_{x_i}
\prod_{i' \neq i} (n_{x_{i'}} - 2) \right),
\end{align}
and this is indeed symmetric in the $n_{x_i}$.

Conversely, suppose $p_X(\bn)$ is symmetric in $\bn$. If $|X| = k \leq 2$
then the result is immediate. Also note that the assertion~(2) for $k
\geq 3$ follows from that for $k=3$ -- since if the distances between any
$3$ distinct points are equal, then $d(x,y) = d(x,y') = d(x',y')$ for all
distinct $x,y,x',y' \in X$ (verifying the remaining cases is easier).
Thus, we suppose henceforth that $|X| = k = 3$. For ease of exposition,
in this proof we denote $d'_{ij} := d(x_i, x_j)$ for $1 \leq i,j \leq 3$.
Also assume by relabeling the $x_i$ (if needed) that $0 < d'_{12} \leq
d'_{13} \leq d'_{23}$. Then
\[
N(\ba_X,\D_X) = \begin{pmatrix}
(n_{x_1} - 1) 2 d'_{12} & n_{x_1} d'_{12} & n_{x_1} d'_{13} \\
n_{x_2} d'_{12} & (n_{x_2} - 1) 2 d'_{12} & n_{x_2} d'_{23} \\
n_{x_3} d'_{13} & n_{x_3} d'_{23} & (n_{x_3} - 1) 2 d'_{13}
\end{pmatrix}.
\]
Since $p_X(\bn) = \det N(\ba_X,\D_X)$ is symmetric in the $n_{x_i}$, we
equate the coefficients of $n_{x_1} n_{x_2}$ and $n_{x_2} n_{x_3}$, to
obtain:
\[
-2 \cdot d'_{13} \cdot (3 (d'_{12})^2) = -2 \cdot d'_{12} \cdot (4
d'_{12} d'_{13} - (d'_{23})^2).
\]
Simplifying this yields: $d'_{12} d'_{13} = (d'_{23})^2$, and since
$d'_{23}$ dominates $d'_{12}, d'_{13}$, the three \textit{distances}
$d'_{12}, d'_{13}, d'_{23}$ are equal. This proves the converse for $|X|
= k = 3$, hence for all $k \geq 3$.
\end{proof}

\section{Real-stability of the blowup-polynomial}

The proofs in Section~\ref{S2} were mostly algebraic in nature: although
they applied to metric spaces, all but the final proof involved no
inequalities. We now show Theorem~\ref{Tstable}: $p_X(\cdot)$ is always
real-stable.

We begin by mentioning some properties with respect to which blowups
behave well. These include iterated blowups, the blowup-polynomial, and
the modified distance matrix $\D_X$ and its positivity. (As
Theorem~\ref{Teuclidean} indicates, the property of being Euclidean is
not such a property.) We first introduce another ``well-behaved'' matrix
$\mathcal{C}_X$ for a finite metric space, parallel to $\D_X$ and the
vector $\ba_X$, which will be useful here and in later sections.

\begin{definition}
Given a finite metric space $X = \{ x_1, \dots, x_k \}$, recall the
vector ${\bf a}_X \in \R^k$ as in~\eqref{Emetricpoly} and define the
symmetric matrix $\mathcal{C}_X \in \R^{k \times k}$, via:
\begin{align}
\begin{aligned}
{\bf a}_X = &\ -2 (d(x_1, X \setminus \{ x_1 \}), \dots, d(x_k, X
\setminus \{ x_k \})) = (-2 d(x, X \setminus \{ x \}))_{x \in X}, \\
\mathcal{C}_X := &\ (-\Delta_{\ba_X})^{-1/2} \D_X
(-\Delta_{\ba_X})^{-1/2},
\end{aligned}
\end{align}
In other words, $-{\bf a}_X$ is the diagonal vector of the modified
distance matrix $\D_X$, and
\begin{equation}\label{Ecmatrix}
(\mathcal{C}_X)_{ij} = \begin{cases}
1, & \text{if } i = j;\\
\displaystyle \frac{d(x_i, x_j)}{2 \sqrt{d(x_i, X \setminus \{ x_i \})
d(x_j, X \setminus \{ x_j \})}}, \qquad & \text{otherwise}.
\end{cases}
\end{equation}
\end{definition}

\begin{lemma}\label{Lwellbehaved}
Fix a finite metric space $(X,d)$ and an integer tuple $\bn = (n_x : x
\in X) \in \mathbb{Z}_{>0}^X$.
\begin{enumerate}
\item Fix a positive integer $m_{xi}$ for each $x \in X$ and $1 \leq i
\leq n_x$, and let ${\bf m} := (m_{xi})_{x,i}$ denote the entire
collection. Then $(X[\bn])[{\bf m}]$ is isometrically isomorphic to
$X[\bn']$, where $\bn' = (\sum_{i=1}^{n_x} m_{xi} : x \in X)$. Here the
$i$th copy of $x$ in $X[\bn]$ is copied $m_{xi}$ times in $(X[\bn])[{\bf
m}]$.

\item In particular, the blowup-polynomial of an iterated blowup is
simply the original blowup-polynomial in a larger number of variables, up
to a constant:
\begin{equation}\label{Eblowup}
p_{X[\bn]}({\bf m}) \equiv p_X(\bn') \prod_{x \in X} a_x^{n_x - 1},
\end{equation}
where the coordinates of $\bn' = (\sum_{i=1}^{n_x} m_{xi} : x \in X)$ are
sums of variables.

\item Now write $X = \{ x_1, \dots, x_k \}$ as above. Then the matrices
$\D_{X[\bn]}, \mathcal{C}_{X[\bn]}$ are both block $k \times k$ matrices,
with $(i,j)$ block respectively equal to
\[
d_{ij} {\bf 1}_{n_{x_i} \times n_{x_j}} \quad \text{and} \quad
c_{ij} {\bf 1}_{n_{x_i} \times n_{x_j}},
\]
where $\D_X = (d_{ij})_{i,j=1}^k, \mathcal{C}_X = (c_{ij})_{i,j=1}^k$.

\item The following are equivalent:
\begin{enumerate}
\item The matrix $\D_X$ is positive semidefinite.

\item The matrix $\D_{X[\bn]}$ is positive semidefinite for some
(equivalently, every) tuple $\bn$ of positive integers.

\item The matrix $\mathcal{C}_X$ is positive semidefinite.

\item The matrix $\mathcal{C}_{X[\bn]}$ is positive semidefinite for some
(equivalently, every) tuple $\bn$ of positive integers.
\end{enumerate}
\end{enumerate}
\end{lemma}

\begin{proof}\hfill
\begin{enumerate}
\item In studying $(X[\bn])[{\bf m}]$, for ease of exposition we write
$Y := X[\bn], Z := (X[\bn])[{\bf m}]$.
Also write $y_{xi}$ for the $i$th copy of $x$ in $Y$, and $z_{xij}$ for
the $j$th copy of $y_{xi}$ in $Z$, with $1 \leq i \leq n_x$ and $1 \leq j
\leq m_{xi}$. We now compute $d_Z(z_{xij},z_{x'i'j'})$, considering three
cases. First if $x \neq x'$ then this equals $d_Y(y_{xi},y_{x'i'}) =
d_X(x,x')$. Next if $x = x'$ but $i \neq i'$ then it equals $d_Y(y_{xi},
y_{xi'}) = 2 d(x, X \setminus \{ x \})$. Finally, suppose $x = x'$ and $i
= i'$ but $j \neq j'$. Then
\[
d_Z(z_{xij},z_{x'i'j'}) = 2 d_Y(y_{xi}, Y \setminus \{ y_{xi} \}),
\]
and it is not hard to show, by considering all distances in $Y$, that
this equals $2 d_X(x, X \setminus \{ x \})$. These three cases reveal
that $d_Z(z_{xij}, z_{x'i'j'})$ equals the distance in $X[\bn']$ between
the copies of $x,x' \in X$, and the proof is complete.

\item We show~\eqref{Eblowup} using the previous part and the next part,
and via Zariski density arguments as in the proof of
Theorem~\ref{Tmonoid}. Define $n_j := n_{x_j}$ in this proof for
convenience. Thus, we work more generally in the setting where $X = \{
x_1, \dots, x_k \}$, but the arrays
\[
\ba_X = (a_{x_1}, \dots, a_{x_k})^T, \qquad \D_X =
(d_{rs})_{r,s=1}^k, \qquad {\bf m} = (m_{j1}, \dots, m_{jn_j})_{j=1}^k
\]
consist of \textit{indeterminates}. Let $K := \sum_{j=1}^k n_j$, and
define $\mathcal{W}_{K \times k}$ to be the block matrix
\[
\mathcal{W} := \begin{pmatrix}
{\bf 1}_{n_1 \times 1} & 0_{n_1 \times 1} & \cdots & 0_{n_1 \times 1} \\
0_{n_2 \times 1} & {\bf 1}_{n_2 \times 1} & \cdots & 0_{n_2 \times 1} \\
\vdots & \vdots & \ddots & \vdots \\
0_{n_k \times 1} & 0_{n_k \times 1} & \cdots & {\bf 1}_{n_k \times 1}
\end{pmatrix}.
\]
Now $\Delta_{\ba_{X[\bn]}} = \diag( a_{x_1} \Id_{n_1}, \dots, a_{x_k}
\Id_{n_k})$, and a straightforward computation (using the next part)
shows that $\D_{X[\bn]} = \mathcal{W} \D_X \mathcal{W}^T$.

Notice that if one works over the field
\[
\mathbb{Q}(\{ a_{x_j}, m_{ji} : 1 \leq j \leq k, \ 1 \leq i \leq n_j \},
\{ d_{rs} : 1 \leq r,s \leq k \}),
\]
then the following polynomial is nonzero:
\begin{equation}\label{Ezariski}
(\det \D_X) \prod_{j=1}^k a_{x_j} \prod_{j=1}^k \prod_{i=1}^{n_j} m_{ji}.
\end{equation}

Thus, we now compute:
\[
p_{X[\bn]}({\bf m}) = \det (\Delta_{a_{X[\bn]}} + \Delta_{\bf m}
\D_{X[\bn]}) = \det (\Delta_{a_{X[\bn]}} + \Delta_{\bf m} \mathcal{W}
\D_X \mathcal{W}^T).
\]
Using~\eqref{Ezariski} and Schur complements, this equals
\[
\det (\Delta_{\bf m}) \cdot \det \begin{pmatrix} \Delta_{\bf m}^{-1}
\Delta_{a_{X[\bn]}} & -\mathcal{W} \\ \mathcal{W}^T & \D_X^{-1}
\end{pmatrix} \det (\D_X).
\]

Using an alternate Schur complement, we expand this latter expression as:
\[
\det (\Delta_{\bf m}) \cdot \det (\Delta_{\bf m}^{-1}) \det
(\Delta_{a_{X[\bn]}}) \det( \D_X^{-1} + \mathcal{W}^T \Delta_{\bf m}
\Delta_{a_{X[\bn]}}^{-1} \mathcal{W}) \cdot \det(\D_X).
\]

Now defining $n'_j := \sum_{i=1}^{n_j} m_{ji}$ as in the assertion, we
have:
\[
\mathcal{W}^T \Delta_{\bf m} \Delta_{a_{X[\bn]}}^{-1} \mathcal{W} =
\diag(a_{x_1}^{-1} n'_1, \dots, a_{x_k}^{-1} n'_k) = \Delta_{a_X}^{-1}
\Delta_{\bn'}.
\]
Thus the above computation can be continued:
\begin{align*}
p_{X[\bn]}({\bf m}) = &\ \det(\Delta_{a_{X[\bn]}}) \det(\D_X^{-1} + 
\Delta_{a_X}^{-1} \Delta_{\bn'}) \det(\D_X)\\
= &\ \prod_{j=1}^k a_{x_j}^{n_j} \cdot \det(\Id_k + \Delta_{a_X}^{-1}
\Delta_{\bn'} \D_X)\\
= &\ \prod_{j=1}^k a_{x_j}^{n_j - 1} \cdot \det(\Delta_{a_X} +
\Delta_{\bn'} \D_X) = p_X(\bn') \prod_{j=1}^k a_{x_j}^{n_j - 1}.
\end{align*}

This proves the result over the function field (over $\mathbb{Q}$) in
which the entries $a_{x_j}, m_{ji}, d_{rs}$ are variables. Now we repeat
the Zariski density arguments as in the proof of Theorem~\ref{Tmonoid},
working this time with the nonzero polynomial given in~\eqref{Ezariski}.
This shows the result over an arbitrary commutative ring -- in
particular, over $\R$.

\item The key observation is that the diagonal entries of $\D_{X[\bn]}$
corresponding to the copies of $x \in X$, all equal $2d_X(x, X \setminus
\{ x \})$, which is precisely the corresponding diagonal entry in $\D_X$.
From this, the assertion for $\D_{X[\bn]}$ is immediate, and that for
$\mathcal{C}_{X[\bn]}$ is also straightforward.

\item We first prove the equivalence for the $\D$-matrices. The preceding
part implies that $\D_X$ is a principal submatrix of $\D_{X[\bn]}$, hence
is positive semidefinite if $\D_{X[\bn]}$ is.
Conversely, given $v \in \R^{n_{x_1} + \cdots + n_{x_k}}$, write $v^T =
(v_1^T, \dots, v_k^T)$, with all $v_i \in \R^{n_{x_i}}$. Let
$w_i := v_i^T {\bf 1}_{n_{x_i}}$, and 
denote by
$w := (w_1, \dots, w_k)^T$
the ``compression'' of $v$. Now compute:
\[
v^T \D_{X[\bn]} v = \sum_{i,j=1}^k v_i^T d_{ij} {\bf 1}_{n_{x_i} \times
n_{x_j}} v_j = \sum_{i,j=1}^k w_i d_{ij} w_j = w^T \D_X w,
\]
and this is non-negative (for all $v$) since $\D_X$ is positive
semidefinite. Hence so is $\D_{X[\bn]}$.

This proves the equivalence for the $\D$-matrices. Now for any metric
space $Y$ (e.g.\ $Y = X$ or $X[\bn]$), the matrix $\mathcal{C}_Y =
(-\Delta_{\ba_Y})^{-1/2} \D_Y (-\Delta_{\ba_Y})^{-1/2}$ is positive
semidefinite if and only if $\D_Y$ is. This concludes the proof. \qedhere
\end{enumerate}
\end{proof}

\begin{remark}
The proof of Lemma~\ref{Lwellbehaved}(2) using Zariski density indicates
a similar, alternate approach to proving the formula for $\det M({\bf A},
{\bf D})$ in Theorem~\ref{Tmonoid}. The difference, now, is that the
rank-one expansion of the matrix ${\bf D}$ is no longer needed, and can
be replaced by the use of the two block-diagonal matrices
\[
\mathcal{W}({\bf p}_1, \dots, {\bf p}_k) := \begin{pmatrix} ({\bf
p}_1)_{n_1 \times 1} & 0_{n_1 \times 1} & \cdots & 0_{n_1 \times 1} \\
0_{n_2 \times 1} & ({\bf p}_2)_{n_2 \times 1} & \cdots & 0_{n_2 \times 1}
\\
\vdots & \vdots & \ddots & \vdots \\
0_{n_k \times 1} & 0_{n_k \times 1} & \cdots & ({\bf p}_k)_{n_k \times 1}
\end{pmatrix}
\]
and a similar matrix $\mathcal{W}({\bf q}_1, \dots, {\bf q}_k)$, so that
$M({\bf A}, {\bf D}) = \diag(\{ A_i \cdot \Id_{n_i} \}) + \mathcal{W}(\{
{\bf p}_i \}) \cdot {\bf D} \cdot \mathcal{W}(\{ {\bf q}_i \})^T$.
\end{remark}

Lemma~\ref{Lwellbehaved}(2) immediately implies the following consequence
(which can also be shown directly):

\begin{corollary}\label{Cblowup}
Fix a finite metric space $(X,d)$. For all integer tuples $\bn \in
\mathbb{Z}_{>0}^X$, the blowup-polynomial of $X[\bn]$ has total degree at
most $|X|$.
\end{corollary}

\noindent In other words, no monomials of degree $|X|+1$ or higher occur
in $p_{X[\bn]}$, for any tuple $\bn$.

We now prove the real-stability of $p_X(\cdot)$:

\begin{proof}[Proof of Theorem~\ref{Tstable}]
We continue to use the notation in the proof of
Theorem~\ref{Tmetricmatrix}, with one addition: for expositional clarity,
in this proof we treat $p_X(\cdot)$ as a polynomial in the complex
variables $z_j := n_{x_j}$ for $j=1,\dots,k$. Thus,
\[
p_X(z_1, \dots, z_k) = \det N(\ba_X,\D_X) = \det (\Delta_{\ba_X} +
\Delta_{\bf z} \D_X),
\]
where $a_j = -2 d(x_j, X \setminus \{ x_j \}) < 0\ \forall j$ and
$\Delta_{\bf z} := \diag(z_1, \dots, z_k)$. We compute:
\begin{align*}
p_X({\bf z}) = &\ \det (\Delta_{\bf z}) \det (\D_X - (-\Delta_{\ba_X})
\Delta_{\bf z}^{-1})\\
= &\ \prod_{j=1}^k z_j \cdot \det (-\Delta_{\ba_X})^{1/2} \det \left(
(-\Delta_{\ba_X})^{-1/2} \D_X (-\Delta_{\ba_X})^{-1/2} - \Delta_{\bf
z}^{-1} \right) \det (-\Delta_{\ba_X})^{1/2}\\
= &\ \det (-\Delta_{\ba_X}) \prod_{j=1}^k z_j \cdot \det \left(
\mathcal{C}_X + \sum_{j=1}^k (-z_j^{-1}) E_{jj} \right),
\end{align*}
where $E_{jj}$ is the elementary $k \times k$ matrix with $(j,j)$ entry
$1$ and all other entries zero.

We now appeal to two facts. The first is a well-known result of
Borcea--Br\"and\'en: \cite[Proposition 2.4]{BB1} (see also~\cite[Lemma
4.1]{Branden}), which says that if $A_1, \dots, A_k, B$ are
equi-dimensional real symmetric matrices, with all $A_j$ positive
semidefinite, then the polynomial
\begin{equation}\label{EBH}
f(z_1, \dots, z_k) := \det \left( B + \sum_{j=1}^k z_j A_j \right)
\end{equation}
is either real-stable or identically zero. The second is the folklore
result that ``inversion preserves stability'' (since the upper half-plane
is preserved under the transformation $z \mapsto -1/z$ of
$\mathbb{C}^\times$). That is, if a polynomial $g(z_1, \dots, z_k)$ has
$z_j$-degree $d_j \geq 1$ and is real-stable, then so is the polynomial
\[
z_1^{d_1} \cdot g(-z_1^{-1}, z_2, \dots, z_k)
\]
(this actually holds for any $z_j$). Apply this latter fact to each
variable of the multi-affine polynomial $f(\cdot)$ in~\eqref{EBH} -- in
which $d_j=1$, $B = \mathcal{C}_X$, and $A_j = E_{jj}\ \forall j$. It
follows that the polynomial
\begin{align*}
z_1 \cdots z_k \cdot f(-z_1^{-1}, \dots, -z_k^{-1}) = &\ \prod_{j=1}^k
z_j \cdot \det \left( \mathcal{C}_X + \sum_{j=1}^k (-z_j^{-1}) E_{jj}
\right)\\
= &\ \det (-\Delta_{\ba_X})^{-1} p_X({\bf z})
\end{align*}
is real-stable, and the proof is complete.
\end{proof}

\begin{remark}
For completeness, we briefly touch upon other notions of stability that
are standard in mathematics (analysis, control theory,
differential/difference equations): Hurwitz stability and Schur
stability. Recall that a real polynomial in one variable is said to be
Hurwitz stable (respectively, Schur stable) if all of its roots lie in
the open left half-plane (respectively, in the open unit disk) in
$\mathbb{C}$. Now the univariate specializations $u_X(n) =
p_X(n,n,\dots,n)$ are not all either Hurwitz or Schur stable. As a
concrete example: in the simplest case of the discrete metric on a space
$X$, Equation~\eqref{Ecomplete} implies that $u_X(n) = (n-2)^{k-1} (n-2 +
kn)$, and this vanishes at $n=2, \frac{2}{k+1}$.
\end{remark}

\section{Combinatorics: Graphs and their partially symmetric
blowup-polynomials}\label{Sgraphs}

We now take a closer look at a distinguished sub-class of finite metric
spaces: unweighted graphs.
In this section, we will show Theorems \ref{Tisom}--\ref{Ttree-blowup}.
To avoid having to mention the same quantifiers repeatedly, we introduce
the following definition (used in the opening section).

\begin{definition}
A \emph{graph metric space} is a finite, simple, connected, unweighted
graph $G$, in which the distance between two vertices is the number of
edges in a shortest path connecting them.
\end{definition}

Every graph metric space $G$ is thus a finite metric space, and so the
results in the previous sections apply to it. In particular, to every
graph metric space $G = (V,E)$ are naturally associated a (to our
knowledge) novel graph invariant
\begin{equation}
p_G(\bn) = p_G(\{ n_v : v \in V \}) := \det (2\Delta_\bn - 2 \Id_V +
\Delta_\bn D_G) = \det (-2 \Id_V + \Delta_\bn \D_G)
\end{equation}
(which we showed is real-stable), as well as its univariate
specialization (which is thus real-rooted)
\begin{equation}
u_G(n) = p_G(n,n,\dots,n) = \det ((2n - 2) \Id_V + n D_G) = \det (-2
\Id_V + n \D_G)
\end{equation}
and its ``maximum root'' $\alpha_{\max}(u_G) \in \R$. Here,
$D_G$ is the distance matrix of $G$ (with zeros on the diagonal) and
$\D_G = D_G + 2 \Id_V$ is the modified distance matrix.

\subsection{Connections to the distance spectrum; $p_G$ recovers $G$}

We begin with an observation (for completeness), which ties into one of
our original motivations by connecting the blowup-polynomial $u_G$ to the
distance spectrum of $G$, i.e.\ to the eigenvalues of the distance matrix
$D_G$. The study of these eigenvalues began with the work of Graham and
Lov\'asz~\cite{GL}, and by now, is a well-developed program in the
literature; see e.g.\ \cite{AH}. Our observation here is the following:

\begin{proposition}\label{Pdistancespectrum}
Suppose $G = (V,E)$ is a graph metric space. A real number $n$ is a root
of the univariate blowup-polynomial $u_G$, if and only if $2n^{-1} - 2$
is an eigenvalue of the distance matrix $D_G$, with the same
multiplicity.
\end{proposition}

Alternately, $\lambda \neq -2$ is an eigenvalue of $D_G$ if and only if
$\frac{2}{2 + \lambda}$ is a root of $u_G$.

\begin{proof}
First note from the definitions that $u_G(0) = \det (-2 \Id_V) \neq 0$.
We now compute:
\[
u_G(n) = \det (-2 \Id_V + n (2 \Id_V + D_G)) = (2n)^{|V|} \det( \Id_V +
\textstyle{\frac{1}{2}} D_G - n^{-1} \Id_V).
\]
Thus, $n$ is a (nonzero) root of $u_G$ if and only if $n^{-1}$ is an
eigenvalue of $\Id_V + \frac{1}{2} D_G$. The result follows from here.
\end{proof}

In the distance spectrum literature, much work has gone into studying the
largest eigenvalue of $D_G$, called the ``distance spectral radius'' in
the literature, as well as the smallest eigenvalue of $D_G$. An immediate
application of Proposition~\ref{Pdistancespectrum} provides an
interpretation of another such eigenvalue:

\begin{corollary}\label{Cdistancespectrum}
The smallest eigenvalue of $D_G$ which is strictly greater than $-2$, is
precisely $\frac{2}{\alpha_{\max}(u_G)} - 2$.
\end{corollary}

\noindent We refer the reader to further discussions about
$\alpha_{\max}(u_G)$ in and around Proposition~\ref{Pmaxroot}.

Following these observations that reinforce our motivating connections
between distance spectra and the blowup-polynomial, we now move on to the
proof of Theorem~\ref{Tisom}. Recall, this result shows that (the
homogeneous quadratic part of) $p_G$ recovers/detects the graph and its
isometries -- but $u_G$ does not do so.

\begin{proof}[Proof of Theorem~\ref{Tisom}]
We prove the various assertions in serial order.
One implication for the first assertion was described just above the
theorem-statement.
Conversely, suppose $p_G(\bn) \equiv p_{\Psi(G)}(\bn)$. Fix vertices $v
\neq w \in V$, and equate the coefficient of $n_v n_w$ on both sides
using Proposition~\ref{Pcoeff}:
\[
(-2)^{|V|-2} \det \begin{pmatrix} 2 & d(v,w) \\ d(v,w) & 2 \end{pmatrix}
= (-2)^{|V|-2} \det \begin{pmatrix} 2 & d(\Psi(v),\Psi(w)) \\
d(\Psi(v),\Psi(w)) & 2 \end{pmatrix}
\]
since $d_G(v, V \setminus \{ v \}) = 1\ \forall v \in V$. Thus
$d(\Psi(v), \Psi(w)) = d(v,w)$ for all $v, w \in V$, so $\Psi$ is an
isometry.

The second assertion is shown as follows. By Proposition~\ref{Pcoeff}, the
vertex set can be obtained from the nonzero monomials $n_v n_w$ (since
every edge yields a nonzero monomial). In particular, $|V|$ is recovered.
Again by Proposition~\ref{Pcoeff}, there is a bijection between the set
of edges $v \sim w$ in $G$ and the monomials $n_v n_w$ in $p_G(\bn)$ with
coefficient $3(-2)^{|V|-2}$. Thus, all quadratic monomials in $p_G(\bn)$
with this coefficient reveal the edge set of $G$ as well.

Finally, to show that $u_G$ does not detect the graph $G$, consider the
two graphs $H,K$ in Figure~\ref{Fig1}.
\begin{figure}[ht]
\definecolor{xdxdff}{rgb}{0.49,0.49,1}
\begin{tikzpicture}
\draw (0,1.5)-- (-1.5,0);
\draw (0,0)-- (0,1.5);
\draw (-1.5,0)-- (0,0);
\draw (0,0)-- (1.5,0);
\draw (1.5,0)-- (3,0);
\draw (3,0)-- (4.5,0);
\fill [color=xdxdff] (-1.5,0) circle (1.5pt);
\draw[color=black] (-1.5,-0.25) node {$2$};
\fill [color=xdxdff] (0,0) circle (1.5pt);
\draw[color=black] (0,-0.25) node {$3$};
\fill [color=xdxdff] (1.5,0) circle (1.5pt);
\draw[color=black] (1.5,-0.25) node {$4$};
\fill [color=xdxdff] (3,0) circle (1.5pt);
\draw[color=black] (3,-0.25) node {$5$};
\fill [color=xdxdff] (4.5,0) circle (1.5pt);
\draw[color=black] (4.5,-0.25) node {$6$};
\fill [color=xdxdff] (0,1.5) circle (1.5pt);
\draw[color=black] (0,1.75) node {$1$};
\draw[color=black] (3,1) node {$H$};
\draw (7.5,1.5)-- (6,0);
\draw (7.5,1.5)-- (7.5,0);
\draw (7.5,1.5)-- (9,0);
\draw (6,0)-- (7.5,0);
\draw (7.5,0)-- (9,0);
\draw (9,0)-- (10.5,0);
\draw (10.5,0)-- (12,0);
\fill [color=xdxdff] (6,0) circle (1.5pt);
\draw[color=black] (6,-0.25) node {$2$};
\fill [color=xdxdff] (7.5,0) circle (1.5pt);
\draw[color=black] (7.5,-0.25) node {$3$};
\fill [color=xdxdff] (9,0) circle (1.5pt);
\draw[color=black] (9,-0.25) node {$4$};
\fill [color=xdxdff] (10.5,0) circle (1.5pt);
\draw[color=black] (10.5,-0.25) node {$5$};
\fill [color=xdxdff] (12,0) circle (1.5pt);
\draw[color=black] (12,-0.25) node {$6$};
\fill [color=xdxdff] (7.5,1.5) circle (1.5pt);
\draw[color=black] (7.5,1.75) node {$1$};
\draw[color=black] (10.5,1) node {$K$};
\end{tikzpicture}
\caption{Two non-isometric graphs on six vertices with co-spectral
blowups}\label{Fig1}
\end{figure}
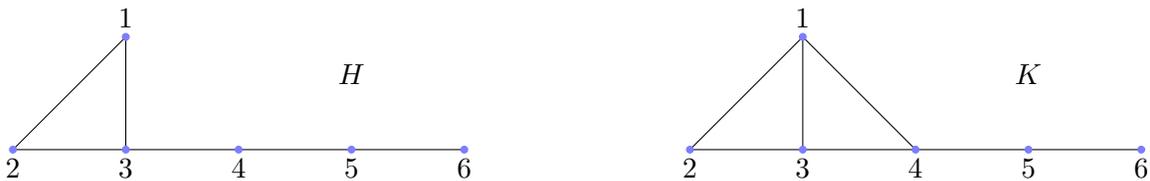
Both graphs have vertex sets $\{ 1, \dots, 6 \}$, and are not isomorphic.
Now define (see Remark~\ref{Rdrury}):
\[
H' := H[(2,1,1,2,1,1)], \qquad K' := K[(2,1,1,1,1,2)].
\]
Then $H', K'$ are not isometric, but a direct computation reveals:
\[
u_{H'}(n) = u_{K'}(n) = -320 n^6 + 3712 n^5 - 10816 n^4 + 10880 n^3 -
1664 n^2 - 2048 n + 256. \qedhere
\]
\end{proof}

\begin{remark}\label{Rdrury}
The graphs $H',K'$ in the preceding proof were not accidental or
providential, but stem from the recent paper~\cite{DL}, which is part of
the literature on exploring which graphs are distance co-spectral (see
the Introduction). In the discussion preceding~\cite[Figure 1]{DL}, the
authors verified that the graphs $H' \not\cong K'$ used in the preceding
proof are indeed distance co-spectral. This result, combined with
Proposition~\ref{Pdistancespectrum}, leads to the above use of $H', K'$
in proving that $u_G$ cannot detect $G$ up to isometry.
\end{remark}

\begin{remark}
As the proof of Theorem~\ref{Tisom} reveals, for any graph metric space
$G = (V,E)$, the Hessian of the blowup-polynomial carries the same
information as the matrix $\D_G \in \mathbb{Z}_{>0}^{V \times V}$:
\begin{equation}
\mathcal{H}(p_G) := ((\partial_{n_v} \partial_{n_{v'}} p_G)({\bf
0}))_{v,v' \in V} = (-2)^{|V|} {\bf 1}_{|V| \times |V|} - (-2)^{|V|-2}
\D_G^{\circ 2},
\end{equation}
where $\D_G^{\circ 2}$ is the entrywise square of the modified distance
matrix $\D_G$.
\end{remark}

\subsection{Complete multipartite graphs via
real-stability}\label{Slorentz}

The next result we show is Theorem~\ref{Tlorentz} (described in the title
of this subsection). Before doing so, we define the three classes of
polynomials alluded to in Corollary~\ref{Clorentz}, as promised there
(and for the self-sufficiency of this paper).

\begin{enumerate}
\item Br\"and\'en--Huh \cite{BH} defined a polynomial $p \in \R[x_1,
\dots, x_k]$ to be {\em Lorentzian} if $p(\cdot)$ is homogeneous of some
degree $d$, has non-negative coefficients, and given any indices $0 \leq
j_1, \dots, j_{d-2} \leq k$, if
\[
g(x_1, \dots, x_k) := \left( \partial_{x_{j_1}} \cdots
\partial_{x_{j_{d-2}} } p \right)(x_1, \dots, x_k),
\]
then the Hessian matrix $\mathcal{H}_g := (\partial_{x_i} \partial_{x_j}
g)_{i,j=1}^k \in \R^{k \times k}$ is Lorentzian.
(This last term means that $\mathcal{H}_g$ is nonsingular and has exactly
one positive eigenvalue.)

\item Suppose $p \in \R[x_1, \dots, x_k]$ has non-negative coefficients.
Gurvits~\cite{Gurvits} defined $p$ to be \textit{strongly log-concave} if
for all $\alpha \in \mathbb{Z}_{\geq 0}^k$, either the derivative
$\displaystyle
\partial^\alpha (p) := \prod_{i=1}^k \partial_{x_i}^{\alpha_i} \cdot p$
is identically zero, or $\log(\partial^\alpha (p))$ is defined and
concave on $(0,\infty)^k$.

\item Suppose $p \in \R[x_1, \dots, x_k]$ has non-negative coefficients.
Anari, Oveis Gharan, and Vinzant~\cite{AGV} defined $p$ to be
\textit{completely log-concave} if for all integers $m \geq 1$ and
matrices $A = (a_{ij}) \in [0,\infty)^{m \times k}$, either the
derivative $\displaystyle
\partial_A (p) := \prod_{i=1}^m \left( \sum_{j=1}^k a_{ij} \partial_{x_j}
\right) \cdot p$
is identically zero, or $\log(\partial_A (p))$ is defined and concave on
$(0,\infty)^k$.
\end{enumerate}

Having written these definitions, we proceed to the main proof.

\begin{proof}[Proof of Theorem~\ref{Tlorentz}]
We prove the cyclic chain of implications:
\[
(4) \ \implies \
(1) \ \implies \
\{ (1), (2) \} \ \implies \
(3) \ \implies \
(2) \ \implies \
(4) \ \Longleftrightarrow \
(5).
\]
We begin with a short proof of $(1) \implies (2)$ via Lorentzian
polynomials from Corollary~\ref{Clorentz}. It was shown in~\cite[pp.\
828--829]{BH} that if $(1)$ holds then $\widetilde{p}_G$ is Lorentzian
(see also \cite[Theorem~6.1]{COSW}), and in turn, this implies $(2)$ by
definition (or by \textit{loc.\ cit.}).

We next show that $(3) \implies (2)$. Observe that
\begin{equation}\label{Erayleigh}
\frac{p_G(-z_1, \dots, -z_k) \cdot (-1)^k}{p_G(-1, \dots, -1) \cdot
(-1)^k} \equiv
\frac{\widetilde{p}_G(1, z_1, \dots, z_k)}{\widetilde{p}_G(1,1,\dots,1)}.
\end{equation}
Now if~$(3)$ holds, then $\widetilde{p}_G(1,1,\dots,1) = (-1)^k p_G(-1,
\dots, -1) > 0$, so the polynomial
\[
(-1)^k p_G(-z_1, \dots, -z_k) = \widetilde{p}_G(1,z_1, \dots, z_k)
\]
has all coefficients non-negative, using~$(3)$ and~\eqref{Erayleigh}.
Since $p_G(\cdot)$ is multi-affine (or by inspecting the form of
$\widetilde{p}_G(\cdot)$), this shows $(3) \implies (2)$.
Now to show $\{ (1), (2) \} \implies (3)$, note that the sum of all
coefficients in $\widetilde{p}_G(\cdot)$ equals
\[
\widetilde{p}_G(1,1,\dots,1) = (-1)^k p_G(-1,\dots,-1),
\]
and by~$(2)$, this dominates the ``constant term'' of $p_G$, i.e.,
\begin{align*}
(-1)^k p_G(-1,\dots,-1) \geq \widetilde{p}_G(1,0,\dots,0) = &\ (-1)^k
p_G(0,\dots,0)\\
= &\ (-1)^k \prod_{v \in V} (-2 d(v, V \setminus \{ v \})) > 0.
\end{align*}
In particular, $(-1)^k p_G(-1,\dots,-1) > 0$, proving a part of~$(3)$.
Hence using~$(2)$ and ~\eqref{Erayleigh}, all coefficients of the
``reflected'' polynomial are non-negative; and the normalization shows that
the coefficients sum to $1$. It remains to show that the ``reflected''
polynomial $p_G(-{\bf z}) / p_G(-1,\dots,-1)$ is real-stable. Once again,
using~\eqref{Erayleigh} and that $(-1)^k p_G(-1,\dots,-1) > 0$, it
suffices to show that $\widetilde{p}_G(1,z_1, \dots, z_k)$ is
real-stable. But this follows from~$(1)$ by specializing to $z_0 \mapsto
1 \in \R$. This finally shows that $(1)$ and $(2)$ together imply $(3)$.

We next show the equivalence of $(4)$ and $(5)$. If $G = K_k$, then $\D_G
= \Id_k + {\bf 1}_{k \times k}$ is positive semidefinite. Hence so is
$\D_{K_k[\bn]}$ for all $\bn$, by Lemma~\ref{Lwellbehaved}(4). The
converse follows from~\cite[Theorem 1.1]{LHWS}, since $\D_G = D_G + 2
\Id_{|V(G)|}$.

Finally, we will show $(2) \implies (4) \implies (1)$. First assume
(2)~$\widetilde{p}_G(\cdot)$ has non-negative coefficients. Fix a subset
$J \subset \{ 1, \dots, k \}$; using Proposition~\ref{Pcoeff}(1), the
coefficient of $z_0^{k - |J|} \prod_{j \in J} z_j$ equals
\[
(-1)^{k - |J|} \cdot \det (\D_G)_{J \times J} \prod_{j \in \{ 1, \dots, k
\} \setminus J} (-d_{jj}) = \det (\D_G)_{J \times J} \prod_{j \in \{ 1,
\dots, k \} \setminus J} 2 d_G(v_j, V \setminus \{ v_j \}).
\]
By the hypotheses, this expression is non-negative for every $J \subset
\{ 1, \dots, k \}$. Hence $\D_G$ has all principal minors non-negative
(and is symmetric), so is positive semidefinite, proving~(4).

Finally, if (4)~$\D_G$ is positive semidefinite, then so is
$\mathcal{C}_G$. As above, write $\mathcal{C}_G =
(-\Delta_{\ba_G})^{-1/2} \D_G (-\Delta_{\ba_G})^{-1/2}$, and compute
using that $-\Delta_{\ba_G}$ is a positive definite diagonal matrix:
\begin{align}\label{Ecomput}
\widetilde{p}_G(z_0, z_1, \dots, z_k) = &\ \det(z_0 \cdot
(-\Delta_{\ba_G}) + \Delta_{\bf z} \D_G)\notag\\
= &\ \det (-\Delta_{\ba_G})^{1/2} \det (z_0 \Id_k + \Delta_{\bf z}
\mathcal{C}_G) \det (-\Delta_{\ba_G})^{1/2}\notag\\
= &\ \det (-\Delta_{\ba_G}) \det (z_0 \Id_k + \Delta_{\bf z}
\mathcal{C}_G).
\end{align}
(As an aside, the second factor in the final expression was termed as the
\textit{multivariate characteristic polynomial of $\mathcal{C}_G$}, in
\cite[Section 4.4]{BH}.)

Now $\mathcal{C}_G$ has a positive semidefinite square root
$\sqrt{\mathcal{C}_G}$ by the hypotheses. We claim that
\[
\det (z_0 \Id_k + \Delta_{\bf z} \sqrt{\mathcal{C}_G}
\sqrt{\mathcal{C}_G}) = \det (z_0 \Id_k + \sqrt{\mathcal{C}_G}
\Delta_{\bf z} \sqrt{\mathcal{C}_G});
\]
this follows by expanding the determinant of
$\begin{pmatrix} z_0 \Id_k & -\sqrt{\mathcal{C}_G} \\ \Delta_{\bf z}
\sqrt{\mathcal{C}_G} & \Id_k \end{pmatrix}$
in two different ways, both using Schur complements. Therefore -- and as
in the proof of Theorem~\ref{Tstable} -- we have:
\begin{align}\label{Emixed}
\begin{aligned}
\widetilde{p}_G(z_0, z_1, \dots, z_k) = &\ \det (-\Delta_{\ba_G}) \det
(z_0 \Id_k + \sqrt{\mathcal{C}_G} \Delta_{\bf z} \sqrt{\mathcal{C}_G})\\
= &\ \det (-\Delta_{\ba_G}) \det \left( z_0 \Id_k + \sum_{j=1}^k z_j
\sqrt{\mathcal{C}_G} E_{jj} \sqrt{\mathcal{C}_G} \right).
\end{aligned}
\end{align}
Now the coefficient of each $z_j, \ j \geq 0$ inside the above
determinant is a positive semidefinite matrix. It follows by
\cite[Proposition 2.4]{BB1} (see the text around~\eqref{EBH}) that
$\widetilde{p}_G(\cdot)$ is real-stable, proving~(1).

This proves the equivalence. The final assertion is immediate from~(4)
via Lemma \ref{Lwellbehaved}(4). 
\end{proof}

\begin{remark}
The assertions in Theorem~\ref{Tlorentz} and Corollary~\ref{Clorentz} are
thus further equivalent to $\mathcal{C}_G$ being a correlation matrix.
Moreover, Theorem~\ref{Tstable} (except the final equivalent assertion
(5)) and Corollary~\ref{Clorentz} hold more generally: for arbitrary
finite metric spaces. We leave the details to the interested reader.
\end{remark}

\begin{remark}
The proof of Theorem~\ref{Tlorentz} also reveals that inspecting the
coefficients in the polynomial $p_G(\cdot)$ or $\widetilde{p}_G(\cdot)$
helps identify the principal minors of $\D_G$ (or $\mathcal{C}_G$) that
are negative, zero, or positive.
\end{remark}

\subsection{Symmetries; monomials with zero coefficients}\label{Rgraphsymm}

The results in this paper are developed with the goal of being used in
proving the main theorems in the opening section. The only exceptions are
one of the Appendices (below), and this present subsection, in which our
goal is to provide further intuition for blowup-polynomials of graph
metric spaces $G$. To do so, we study a concrete family of graph metric
spaces $K^{(l)}_k$ -- for which we compute the blowup-polynomials and
reveal connections to symmetric functions. In addition, these
computations lead to the study of certain monomials whose coefficients in
$p_G$ vanish -- this provides intuition for proving
Theorem~\ref{Ttree-blowup}.

We begin from the proof of Theorem~\ref{Tisom}, which shows that the
blowup-polynomial of $G$ is a \textit{partially symmetric polynomial}, in
the sense of being invariant under the subgroup ${\rm Isom}(G)$ (of
isometries of $G$) of $S_{V(G)}$ (the permutations of $V(G)$). For
instance, $p_G(\bn)$ is symmetric in $\bn$ for $G$ a complete graph;
whereas $p_G$ possesses ``cyclic'' symmetries for $G$ a cycle; and so on.

In addition to these global symmetries (i.e., isometries) of $G$, there
also exist ``local symmetries''. For example, suppose $G = P_k$ is a path
graph, with vertex $x_i$ adjacent to $x_{i \pm 1}$ for $1 < i < k$. For
any $1 < i \leq j < k$, the coefficient of the monomials
\[
n_{x_{i-1}} n_{x_i} \cdots n_{x_j} \qquad \text{and} \qquad
n_{x_i} n_{x_{i+1}} \cdots n_{x_{j+1}}
\]
are equal, by Proposition~\ref{Pcoeff}(3). A similar result holds more
generally for banded graphs with bandwidth $d \geq 1$, in which a vertex
$x_i$ is adjacent to $x_j$ if and only if $0 < |i-j| < d$ and $1 \leq i,j
\leq k$.

\begin{example}
We now use this principle of symmetry to compute the blowup-polynomial
for a two-parameter family of graph metric spaces
\[
G = K_k^{(l)}, \qquad 0 \leq l \leq k-2.
\]
(This will shortly be followed by another application: a sufficient
condition for when certain monomials do not occur in $p_G(\bn)$.)
Here, $K_k^{(l)}$ denotes the complete graph $K_k$ from which the edges
$(1,2), \dots, (1,l+1)$ have been removed. This leads to three types of
vertices:
\[
\{ 1 \}, \qquad \{ 2, \dots, l+1 \}, \qquad \{ l+2, \dots, k \},
\]
and correspondingly, the isometry group ${\rm Isom}(K_k^{(l)}) \cong S_l
\times S_{k-l-1}$. Notice that the vertices $2, \dots, k$ form a complete
induced subgraph of $K_k^{(l)}$.

The graphs $K_k^{(l)}$ are all chordal (i.e., do not contain an induced
$m$-cycle for any $m \geq 4$), and include as special cases: complete
graphs (for $l=0$) as well as complete graphs with one pendant vertex
(for $l=k-2$). The ``almost complete graph'' $K_k^{(1)}$ (missing exactly
one edge) is another special case, important from the viewpoint of matrix
positivity: it was crucially used in \cite{GKR-critG} to compute a graph
invariant which arose out of analysis and positivity, for every chordal
graph. This was termed the \textit{critical exponent} of a graph in
\cite{GKR-critG}, and seems to be a novel graph invariant.

By the remarks above, the blowup-polynomial in the $n_{v_i}$ (which we
will replace by $n_i,\ 1 \leq i \leq k$ for convenience) will be
symmetric separately in $\{ n_2, \dots, n_{l+1} \}$ and in $\{ n_{l+2},
\dots, n_k \}$. In particular, since the polynomial is multi-affine in
the $n_i$, the only terms that appear will be of the form
\begin{equation}\label{Ecoeff}
n_1^{\varepsilon_1} e_r(n_2, \dots, n_{l+1}) e_s(n_{l+2}, \dots, n_k),
\end{equation}
where $\varepsilon_1 = 0$ or $1$, and $e_r(\cdot)$ is the elementary
symmetric (homogeneous, multi-affine, and in fact real-stable) polynomial
for every $r \geq 0$ (with $e_0(\cdot) := 1$).
\end{example}

With this preparation, we can state and prove the following result for
the graphs $K_k^{(l)}$.

\begin{proposition}\label{PKkl}
Fix non-negative integers $k,l$ such that $0 \leq l \leq k-2$. With
$K_k^{(l)}$ as defined above, and denoting $n_{v_i}$ by $n_i$ for
convenience (with $1 \leq i \leq k$), we have
\begin{align*}
p_{K_k^{(l)}}(\bn) = &\ \sum_{r=0}^l \sum_{s=0}^{k-l-1} \left[
(-2)^{k-r-s} (1 + r + s) \right] e_r(n_2, \dots, n_{l+1}) e_s(n_{l+2},
\dots, n_k)\\
&\ + n_1 \sum_{r=0}^l \sum_{s=0}^{k-l-1} \left[ (-2)^{k-r-s-1} (1 -r)
(s+2) \right] e_r(n_2, \dots, n_{l+1}) e_s(n_{l+2}, \dots, n_k).
\end{align*}
\end{proposition}

Notice that setting $n_1=0$, we obtain the blowup-polynomial of the
complete graph $K_{k-1}$, in the variables $n_2, \dots, n_k$. This
clearly equals (via $r+s \leadsto s$):
\[
\sum_{s=0}^{k-1} (-2)^{k-s} (1+s) e_s(n_2, \dots, n_k),
\]
and it also equals the expression in~\eqref{Ecomplete} (modulo
relabelling of the variables) since the underlying metric spaces are
isometric. A similar computation holds upon working with $l=0$.

\begin{proof}[Proof of Proposition~\ref{PKkl}]
We begin by exploiting the symmetries in $K_k^{(l)}$, which imply that
given a subset $I \subset \{ 1, \dots, k \}$, the coefficient of
$\prod_{i \in I} n_i$ depends only on the three integers
\[
\varepsilon_1 := {\bf 1}(1 \in I), \qquad r := \# (I \cap \{ 2, \dots,
l+1 \}), \qquad s := \# (I \cap \{ l+2, \dots, k \}).
\]

Now using~\eqref{Ecoeff}, it follows that the
blowup-polynomial indeed has the desired form:
\begin{align*}
p_{K_k^{(l)}}(\bn) = &\ \sum_{r=0}^l \sum_{s=0}^{k-l-1} a_{0,r,s}
e_r(n_2, \dots, n_{l+1}) e_s(n_{l+2}, \dots, n_k)\\
&\ + n_1 \sum_{r=0}^l \sum_{s=0}^{k-l-1} a_{1,r,s} e_r(n_2, \dots,
n_{l+1}) e_s(n_{l+2}, \dots, n_k),
\end{align*}
for some coefficients $a_{\varepsilon_1,r,s} \in \R$. It remains to
compute these coefficients, and we begin with the coefficients
$a_{0,r,s}$. By a computation akin to the proof of
Proposition~\ref{Pcoeff}(1), these are obtained by specializing $n_1 =
0$, which leads to a power of $(-2)$ times a principal minor of
$\D_{K_k^{(l)}}$ not involving its first row or column. But this is the
determinant of a principal submatrix of
\[
\Id_{k-1} + {\bf 1}_{(k-1) \times (k-1)}
\]
of size $(r+s) \times (r+s)$, and so a direct computation implies the
desired formula:
\[
a_{0,r,s} = (-2) \cdot (-2)^{k-r-s-1} \cdot (1 + r + s).
\]

It remains to compute $a_{1,r,s}$, which equals $(-2)^{k-r-s-1}$ times
the determinant of the block matrix
\[
\D_{r,s} := \begin{pmatrix}
2 & 2 {\bf 1}_{r \times 1}^T & {\bf 1}_{s \times 1}^T \\
2 {\bf 1}_{r \times 1} & \Id_r + {\bf 1}_{r \times r} & {\bf 1}_{r \times
s} \\
{\bf 1}_{s \times 1} & {\bf 1}_{s \times r} & \Id_s + {\bf 1}_{s \times
s} \end{pmatrix} 
= \begin{pmatrix} 2 & v^T \\ v & C_{r,s} \end{pmatrix}
\in \R^{(1+r+s) \times (1+r+s)},
\]
where $v := (2 {\bf 1}_{r \times 1}^T \quad {\bf 1}_{s \times
1}^T)^T \in \R^{r+s}$, and $C_{r,s} = \Id_{r+s} + {\bf 1}_{(r+s) \times
(r+s)}$ denotes the principal submatrix of $\D_{r,s}$ obtained by
removing its first row and column. Now $C_{r,s}^{-1} = \Id - (1+r+s)^{-1}
{\bf 1}_{(r+s) \times (r+s)}$, so using Schur complements,
\begin{align*}
&\ \det \D_{r,s}\\
= &\ \det C_{r,s} \cdot \det \left( 2 - v^T C_{r,s}^{-1} v \right)\\
= &\ 2(1+r+s) - \begin{pmatrix} 2 {\bf 1}_{r
\times 1} \\ {\bf 1}_{s \times 1} \end{pmatrix}^T
\begin{pmatrix} (1+r+s) \Id_r - {\bf 1}_{r \times r} & - {\bf 1}_{r
\times s} \\ - {\bf 1}_{s \times r} & (1+r+s) \Id_s - {\bf 1}_{s \times
s} \end{pmatrix}
\begin{pmatrix} 2 {\bf 1}_{r \times 1} \\ {\bf 1}_{s \times 1}
\end{pmatrix},
\end{align*}
and a straightforward (if careful) calculation reveals this quantity to
equal $(1-r)(s+2)$.
\end{proof}

\begin{example}\label{Exbipartite}
As an additional example, we compute the blowup-polynomial for several
other graphs at once. The path graph $P_3$ is a special case of the star
graphs $K_{1,k-1}$, and in turn, these as well as the cycle graph $C_4$
are special cases of complete bipartite graphs $K_{r,s}$. As $K_{r,s} =
K_2[(r,s)]$ is a blowup, we can use Lemma~\ref{Lwellbehaved}(2) and
Equation~\eqref{Ecomplete} for $k=2$, to obtain:
\begin{align}\label{Ebipartite}
\begin{aligned}
&\ p_{K_{r,s}}(n_1, \dots, n_r; m_1, \dots, m_s)\\
= &\ (-2)^{r+s-2} \left( 3 \sum_{i=1}^r n_i \cdot \sum_{j=1}^s m_j - 4
\sum_{i=1}^r n_i - 4 \sum_{j=1}^s m_j + 4 \right).
\end{aligned}
\end{align}
\end{example}

As one observes by visual inspection, the coefficient in
Proposition~\ref{PKkl} of $n_1 n_2$ times any monomial in the
(``type-$3$'') node-variables $n_{l+2}, \dots, n_k$, vanishes. Similarly,
in~\eqref{Ebipartite} there are many coefficients that vanish -- in fact,
every coefficient of total degree at least $3$. These facts can be
explained more simply, by the following result about zero terms in the
blowup-polynomial:

\begin{proposition}\label{Pzeroterms}
Suppose $G,H$ are graph metric spaces, and $\bn \in
\mathbb{Z}_{>0}^{V(G)}$ a tuple of positive integers, such that $G[\bn]$
isometrically embeds as a subgraph metric space inside $H$. Also suppose
$n_v \geq 2$ for some $v \in V(G)$, and $v_1, v_2 \in G[\bn]$ are copies
of $v$. Then for every subset of vertices $\{ v_1, v_2 \} \subset S
\subset V(G[\bn])$, the coefficient of $\prod_{s \in S} n_s$ in
$p_H(\cdot)$ is zero.
\end{proposition}

For example, for $H$ the path graph $P_4$ with vertices $a - b - c - d$,
the coefficients of $n_a n_c, n_a n_b n_c$, and $n_b n_d, n_b n_c n_d$ in
$p_H(\bn)$ are all zero, since the path subgraphs $a-b-c$ and $b-c-d$ are
both isomorphic to the graph blowup $K_2[(1,2)]$.

This result also extends to arbitrary finite metric spaces.

\begin{proof}
By Proposition~\ref{Pcoeff}, the coefficient of $\bn^S$ (whose meaning is
clear from context) in $p_H(\{ n_w : w \in V(H) \})$ is a scalar times
$\det (\D_H)_{S \times S}$. Since $G[\bn]$ is a metric subspace of $H$,
it suffices to show that $\det (\D_{G[\bn]})_{S \times S} = 0$, since
this matrix agrees with $(\D_H)_{S \times S}$. But since $v_1, v_2 \in
V(G[\bn])$ are copies of $v \in V(G)$, the matrix $\D_{G[\bn]}$ has two
identical rows by Lemma~\ref{Lwellbehaved}(3). It follows that $\det
(\D_{G[\bn]})_{S \times S} = 0$, and the proof is complete.
\end{proof}

\subsection{The tree-blowup delta-matroid}

We conclude this section by proving Theorem~\ref{Ttree-blowup} about the
delta-matroid $\mathcal{M}'(T)$ for every tree, which seems to be
unstudied in the literature. We then explore
(a)~if this delta-matroid equals the blowup delta-matroid
$\mathcal{M}_{\D_T}$; and
(b)~if this construction can be extended to arbitrary graphs.

To motivate the construction of $\mathcal{M}'(T)$, which we term the
\textit{tree-blowup delta-matroid}, we begin by applying
Proposition~\ref{Pzeroterms} with $G = P_2 = K_2$ and $G[\bn] = P_3$. An
immediate consequence is:

\begin{corollary}\label{Cp3}
If $a,b,c \in V(H)$ are such that $b$ is adjacent to $a,c$ but $a,c$ are
not adjacent, then $n_a n_c$ and $n_a n_b n_c$ have zero coefficients in
$p_H(\bn)$.
\end{corollary}

In light of this, we take a closer look at $\D_X$ for $X = P_k$ a path
graph. Given an integer $k \geq 1$, the \emph{path graph} $P_k$ has
vertex set $\{ 1, \dots, k \}$, with vertices $i,j$ adjacent if and only
if $|i-j|=1$.

Recall from \cite{Bouchet2} that the set system $\mathcal{M}_{\D_X}$
(defined a few lines before Corollary~\ref{Cdeltamatroid}) is a linear
delta-matroid for every finite metric space -- in particular, for every
graph metric space, such as $P_k$. It turns out (by inspection) that
Corollary~\ref{Cp3} describes \textit{all} of the principal minors of
$\D_{P_k}$ which vanish -- i.e., the complement of the delta-matroid
$\mathcal{M}_{\D_{P_k}}$ in the power-set of $\{ 1, \dots, k \}$ -- for
\textit{small} values of $k$. While this may or may not hold for general
$k$, the same construction still defines a delta-matroid, and one that is
hitherto unexplored as well:

\begin{proposition}\label{Ppk}
Given an integer $k \geq 3$, define the set system
\[
\mathcal{B}(P_k) := 2^{\{ 1, \dots, k \}} \setminus \left\{ \ \{ i, i+1,
i+2 \}, \ \ \{ i, i+2 \} \ : \ 1 \leq i \leq k-2 \right\}.
\]
Then $\mathcal{B}(P_k)$ is a delta-matroid.
\end{proposition}

We now strengthen this result to the case of arbitrary trees (for
completeness, recall that a tree is a finite connected graph without a
cycle $C_n$ for $n \geq 3$). We begin with a few basic observations on
trees, which help in the next proofs. Every pair of vertices in $T$ is
connected by a unique path in $T$. Moreover, every connected sub-graph of
$T$ is a tree, and every (nonempty) set $I$ of vertices of $T$ is
contained in a unique smallest sub-tree $T(I)$, called its
\textit{Steiner tree}. We also recall that a \textit{leaf}, or a
\textit{pendant vertex}, is any vertex with degree one, i.e., adjacent to
exactly one vertex.

\begin{definition}
Let $T = (V,E)$ be a finite connected, unweighted, tree with the (unique)
edge-distance metric.
We say that a subset of vertices $I \subset V$ is \emph{infeasible} if
there exist vertices $v_1 \neq v_2$ in $I$, such that the Steiner tree
$T(I)$ has $v_1, v_2$ as leaves, both adjacent to the same (unique)
vertex.
\end{definition}

With this terminology at hand, Theorem~\ref{Ttree-blowup} asserts that
the feasible subsets of $V$ form a delta-matroid.
Note, if $T = P_k$ is a path graph then $\mathcal{M}'(T)$ equals the
delta-matroid $\mathcal{B}(P_k)$ above.

The proof of Theorem~\ref{Ttree-blowup} requires a preliminary result,
which characterizes when a graph $G$ is a nontrivial blowup, and which
then connects this to the distance spectrum of $G$.

\begin{proposition}\label{Pblowup}
Suppose $G = (V,E)$ is a graph metric space. Then each of the following
statements implies the next:
\begin{enumerate}
\item $G$ is a nontrivial blowup, i.e., a blowup of a graph metric space
$H$ with $|V(H)| < |V(G)|$.

\item $G$ contains two vertices $v,w$ with the same set of neighbors. (In
particular, $d_G(v,w) = 2$.)

\item $-2$ is an eigenvalue of the distance matrix $D_G$.

\item The blowup-polynomial has total degree strictly less than $|V|$.
\end{enumerate}
In fact $(1) \Longleftrightarrow (2) \implies (3) \Longleftrightarrow
(4)$, but all four assertions are not equivalent.
\end{proposition}

\begin{proof}
We show the implications $(2) \implies (1) \implies (2) \implies (3)
\implies (4) \implies (3)$. If~(2) holds, and $H$ is the induced subgraph
of $G$ on $V(G) \setminus \{ w \}$, then it is easy to see that $G =
H[\bn]$, where $n_{v'} = 1$ for $v' \neq v$, $n_v = 2$, and $w$ is the
other copy of $v$. This shows~(1); conversely, if~(1) holds then there
exist two vertices $v \neq w \in V(G)$ that are copies of one another.
But then $v,w$ share the same neighbors, so $0 < d_G(v,w) \leq 2$. If
$d_G(v,w) = 1$ then $v,w$ are neighbors of each other but not of
themselves, which contradicts the preceding sentence. This shows that
$(1) \implies (2)$.

Now suppose~(2) holds, and hence so does~(1). By
Lemma~\ref{Lwellbehaved}(3), $\D_G = D_G + 2 \Id_V$ has two identical
rows, so is singular. This shows~(3). (This implication is also mentioned
in \cite[Theorem 2.34]{AH}.) Finally,~(3) holds if and only if $\D_G$ is
singular as above. By Proposition~\ref{Pcoeff}(1), this is if and only if
the coefficient of the unique top-degree monomial $\prod_{v \in V(G)}
n_v$ in $p_G(\bn)$ vanishes, which shows that $(3) \Longleftrightarrow
(4)$.

To show~(3) does not imply~(2), consider the path graph $P_9$ (with edges
$\{ i, i+1 \}$ for $0 < i < 9$). It is easy to check that $\det \D_{P_9}
= 0$, so that~(3) holds; and also that~(2) does not hold.
\end{proof}

Specializing Proposition~\ref{Pblowup} to trees, we obtain:

\begin{corollary}\label{Ctree-blowup}
A tree $T$ is a blowup of a graph $G$ if and only if
(a)~$G$ is a sub-tree of $T$, and
(b)~the only vertices of $G$ which are ``copied'' are a set of leaves of
$G$.
\end{corollary}

\begin{proof}
One way is easily shown, and for the ``only if'' part, the key
observation is that if a vertex adjacent to at least two others is
copied, then this creates a $4$-cycle in the blowup. (An equally short
proof is that this result is a special case of Proposition~\ref{Pblowup}
$(1) \Longleftrightarrow (2)$.)
\end{proof}

\begin{proof}[Proof of Theorem~\ref{Ttree-blowup}]
If $T$ has two nodes (so $T = K_2$), then $\mathcal{M}'(T) = 2^V$, which
is a delta-matroid. Thus, suppose henceforth that $|V| \geq 3$.
Since $\mathcal{M}'(T)$ contains all singleton subsets of $V$, the only
nontrivial step is to verify the symmetric exchange axiom. Suppose $A
\neq B \in \mathcal{M}'(T)$ and $x \in A \Delta B$. We divide the
analysis into two cases; in what follows, given a subset $I \subset V$,
we will denote by $T_I$ the subgraph of $T$ induced on the vertex set
$I$, and by $T(I)$ the Steiner tree of $I$ (as above).
\begin{enumerate}
\item $x \in B \setminus A$.

If $A \sqcup \{ x \} \in \mathcal{M}'(T)$, then we simply choose $y=x$.
Otherwise $A \sqcup \{ x \}$ is infeasible whereas $A$ is not (i.e., $A$
is feasible). In particular, $x \not\in T(A)$ since otherwise $T(A \sqcup
\{ x \}) = T(A)$. Now using Corollary~\ref{Ctree-blowup}, this yields a
unique $v \in A$ such that $v,x$ are leaves in the sub-tree $T(A \sqcup
\{ x \})$. Also denote their (unique) common adjacent vertex by $a \in
T(A \sqcup \{ x \})$. 

We now proceed. If $v \not\in B$, then compute using $y = v$:
\[
A \Delta \{ x, y \} = (A \setminus \{ v \}) \sqcup \{ x \}.
\]
Since $v,x$ were copies, removing $v$ and adding $x$ produces an
isomorphic graph: $T_{(A \setminus \{ v \}) \sqcup \{ x \}} \cong T_A$,
which is indeed in $\mathcal{M}'(T)$, as desired.

Otherwise $v \in B$. In this case, $T(B)$ contains $v,x$ and hence $a$ as
well. If now $v,x$ are leaves in $T(B)$ then $B$ is infeasible, which
contradicts the assumption. Hence there exists $y \in B$ which is
separated from $a \in V$ by (exactly) one of $v,x \in B$. Clearly, $y
\not\in A \sqcup \{ x \}$; now note that $A \Delta \{ x, y \} = A \sqcup
\{ x, y \}$, and this belongs to $\mathcal{M}'(T)$ by
Corollary~\ref{Ctree-blowup} (and the uniqueness of the leaves $v,x$ with
a common adjacent vertex).

\item The other case is when $x \in A \setminus B$.

Once again, there are two cases, analogous to the two cases above. First
if $A \setminus \{ x \} \in \mathcal{M}'(T)$, then we simply choose
$y=x$. Otherwise $A \setminus \{ x \}$ is infeasible whereas $A$ is not.
Using Corollary~\ref{Ctree-blowup}, this yields unique $v,w \in A$ such
that:
\begin{itemize}
\item $v,w$ are leaves in $T(A \setminus \{ x \})$,
\item $v,w$ are both adjacent to a (unique) vertex $a \in T(A \setminus
\{ x \})$, and 
\item $v$ separates $x$ from $w$ in $T(A)$.
\end{itemize}

There are now two sub-cases. First if $v,w \in B$, then $T(B)$ contains
$v,w$, hence $a$. As $B \in \mathcal{M}'(T)$, there again exists $y \in
B$ which is separated from $a \in V$, now by (exactly) one of $v,w$. But
then $A \Delta \{ x, y \} = (A \setminus \{ x \}) \sqcup \{ y \}$, and
this is in $\mathcal{M}'(T)$ as in the preceding case, by
Corollary~\ref{Ctree-blowup} (and the uniqueness of the leaves $v,w$ with
a common adjacent vertex).

The final sub-case is when not both $v,w$ lie in $B$. Choose an element
$y \in \{ v, w \} \setminus B \subset A \setminus B$. Now $A \Delta \{ x,
y \} = (A \setminus \{ x \}) \setminus \{ y \}$, and by the uniqueness of
the leaves $v,w$ (with a common adjacent vertex), this set lacks two
leaves at distance $2$ from each other in $T$. Therefore $A \Delta \{ x,
y \} \in \mathcal{M}'(T)$, as desired. \qedhere
\end{enumerate}
\end{proof}

As mentioned above, the delta-matroids $\mathcal{M}_{\D_{P_k}} =
\mathcal{M}'(P_k) = \mathcal{B}(P_k)$ for small values of $k$. It is
natural to seek to know if this result holds for all path graphs, and
perhaps even more generally. It turns out that this is false, and the
situation is more involved even for path graphs $P_k$ with $k \geq 9$:

\begin{proposition}\label{Ppath}
Suppose $k \geq 3$ is an integer. The blowup delta-matroid
$\mathcal{M}_{\D_{P_k}}$ of the graph metric space $P_k$ coincides with
the tree-blowup delta-matroid $\mathcal{B}(P_k)$, if and only if $k \leq
8$.
\end{proposition}

\begin{proof}
An explicit (and longwinded) inspection shows that
$\mathcal{M}_{\D_{P_k}} = \mathcal{B}(P_k)$ for $3 \leq k \leq 8$.
Another direct computation shows that $\det \D_{P_9} = 0$. Hence by
Proposition~\ref{Pcoeff}(2), the coefficient of $n_1 n_2 \cdots n_9$ in
$p_{P_k}(\bn)$ also vanishes, for all $k \geq 9$. Using
Proposition~\ref{Pcoeff}(3), it follows that
\[
\{ i, i+1, \dots, i+8 \} \not\in \mathcal{M}_{\D_{P_k}}, \qquad \forall k
\geq 9, \ 1 \leq i \leq k-8.
\]
But these sets all lie in $\mathcal{B}(P_k)$, so $\mathcal{B}(P_k)
\supsetneq \mathcal{M}_{\D_{P_k}}$ for $k \geq 9$.
\end{proof}

As also promised above, we next explore the question of extending
Theorem~\ref{Ttree-blowup} from trees to arbitrary graphs. The key
observation is that the equivalence $(1) \Longleftrightarrow (2)$ in
Proposition~\ref{Pblowup} extends Corollary~\ref{Ctree-blowup}. This
suggests how to define a set system in $2^{V(G)}$ for every graph metric
space $G$, which specializes to the delta-matroid $\mathcal{M}'(T)$ when
$G = T$ is a tree.

\begin{definition}
Suppose $G = (V,E)$ is a graph metric space. We say that a subset $I
\subset V$ is
\begin{enumerate}
\item \textit{infeasible of the first kind} if there exist vertices $v_1
\neq v_2 \in I$ and a subset $I \subset \widetilde{I} \subset V$, such
that:
(a)~the induced subgraph $G(\widetilde{I})$ on $\widetilde{I}$ is
connected in $G$, and
(b)~$v_1, v_2$ have the same set of neighbors in $G(\widetilde{I})$.

\item \textit{infeasible of the second kind} if there exist vertices $v_1
\neq v_2 \in I$ and a subset $I \subset \widetilde{I} \subset V$, such
that:
(a)~the induced subgraph $G(\widetilde{I})$ on $\widetilde{I}$ is a
\textit{metric subspace} of $G$ (hence connected), and
(b)~$v_1, v_2$ have the same set of neighbors in $G(\widetilde{I})$.
\end{enumerate}
Also define $\mathcal{M}'_1(G)$ (respectively, $\mathcal{M}'_2(G)$) to
comprise all subsets of $V$ that are not infeasible of the first
(respectively, second) kind.
\end{definition}

For instance, if $G = T$ is a tree, then it is not hard to see that
$\mathcal{M}'_1(T) = \mathcal{M}'_2(T) = \mathcal{M}'(T)$, which was
studied in Theorem~\ref{Ttree-blowup}. Thus, a question that is natural
given that theorem, is \textit{whether or not $\mathcal{M}'_1(G)$ and/or
$\mathcal{M}'_2(G)$ is a delta-matroid for every graph $G$?}

This question is also natural from the viewpoint of the
blowup-polynomial, in that if $I \subset V$ is infeasible of the second
kind, then the coefficient in $p_G(\bn)$ of the monomial $\prod_{i \in I}
n_i$ is zero by Proposition~\ref{Pzeroterms}. Nevertheless, our next
result shows that this question has a negative answer, already for a
graph on seven vertices. Thus, the construction of $\mathcal{M}'(T)$ does
not extend to arbitrary graph metric spaces in either of the above two
ways:

\begin{proposition}
There exists a graph $G$ such that neither $\mathcal{M}'_1(G)$ nor
$\mathcal{M}'_2(G)$ is a delta-matroid.
\end{proposition}

\begin{proof}
Let $G$ denote the graph in Figure~\ref{Fig2} on the vertex set $V = \{
u, v_1, v_2, w_1, w_2, x, z \}$.
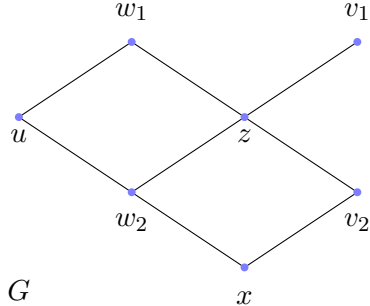
\begin{figure}[ht]
\definecolor{xdxdff}{rgb}{0.49,0.49,1}
\begin{tikzpicture}
\draw (0,1)-- (-1.5,0);
\draw (0,1)-- (1.5,0);
\draw (0,-1)-- (-1.5,0);
\draw (0,-1)-- (1.5,0);
\draw (1.5,0)-- (3,1);
\draw (1.5,0)-- (3,-1);
\draw (1.5,-2)-- (3,-1);
\draw (1.5,-2)-- (0,-1);
\draw[color=black] (-1.5,-2.3) node {$G$};
\fill [color=xdxdff] (-1.5,0) circle (1.5pt);
\draw[color=black] (-1.5,-0.25) node {$u$};
\fill [color=xdxdff] (0,-1) circle (1.5pt);
\draw[color=black] (0,-1.4) node {$w_2$};
\fill [color=xdxdff] (1.5,0) circle (1.5pt);
\draw[color=black] (1.5,-0.25) node {$z$};
\fill [color=xdxdff] (3,1) circle (1.5pt);
\draw[color=black] (3,1.4) node {$v_1$};
\fill [color=xdxdff] (3,-1) circle (1.5pt);
\draw[color=black] (3,-1.4) node {$v_2$};
\fill [color=xdxdff] (1.5,-2) circle (1.5pt);
\draw[color=black] (1.5,-2.4) node {$x$};
\fill [color=xdxdff] (0,1) circle (1.5pt);
\draw[color=black] (0,1.4) node {$w_1$};
\end{tikzpicture}
\caption{A graph on seven vertices}\label{Fig2}
\end{figure}
The definitions show that $A := V$ and $B := \{v_2 \}$ both lie in
$\mathcal{M}'_1(G) \cap \mathcal{M}'_2(G)$, and clearly, $x \in A \Delta
B = A \setminus B$. Moreover, for all $y \in V$, one can verify that $A
\Delta \{ x, y \} = A \setminus \{ x, y \}$ (or $A \setminus \{ x \}$ if
$y=x$) is indeed infeasible of the second kind, hence of the first.
(This uses that in the induced subgraph on $A \setminus \{ x, y \}$,
either $v_1, v_2$ are both present, or $w_1, w_2$ are, and then they are
copies of one another.) Hence the symmetric exchange axiom fails for both
$\mathcal{M}'_1(G)$ and for $\mathcal{M}'_2(G)$, proving the result.
\end{proof}

\section{Metric geometry: Euclidean blowups}

In this section we explore the blowup-polynomial from the viewpoint of
metric geometry, and prove the final outstanding theorem.
Specifically, we understand which blowups $X[\bn]$ of a given finite
metric space $X$ isometrically embed into some Euclidean space $(\R^r, \|
\cdot \|_2)$.

\begin{proof}[Proof of Theorem~\ref{Teuclidean}]
If $k=1$ then the result is immediate, since $X[n_1]$ comprises the
vertices of a Euclidean simplex, for any $n_1 \geq 2$. Thus, we suppose
henceforth that $k \geq 2$.

We first show $(2) \implies (1)$. Let $j_0 \in [1,k] \setminus \{ j \}$
be such that $x_{j_0} = v \in V$ is the closest point in $V$ to $x_j$,
Now a straightforward computation reveals that $x'_j := 2 x_{j_0} - x_j
\in \R^r$ serves as a blowup of $x_j$ in $X[\bn]$. This proves (1).

The proof of the reverse implication $(1) \implies (2)$ is in steps.
We first suppose $n_{x_1}, n_{x_2} \geq 2$ and arrive at a contradiction.
Indeed, now the metric subspace $Y[(2,2)] \subset X[\bn]$ is also
Euclidean, where $Y = ( x_1, x_2 )$. Rescale the metric such that
$d(x_1,x_2) = 1$, and let $y_i$ be the blowup of $x_i$ in $Y[(2,2)]$ for
$i=1,2$. Then the modified Cayley--Menger matrix of $Y[(2,2)]$ with
respect to $(x_1, y_1, x_2, y_2)$ is
\[
A = \begin{pmatrix} 8 & 4 & 4 \\ 4 & 2 & -2 \\ 4 & -2 & 2 \end{pmatrix}.
\]
As $\det A < 0$, it follows by Theorem~\ref{Tschoenberg} that $Y[(2,2)]$,
and hence $X[\bn]$, is not Euclidean.

Next, we suppose $n_{x_2} \geq 3$ and again arrive at a contradiction --
this time by considering the Euclidean subspace $Y[(1,3)] \subset
X[\bn]$, where $Y = (x_1, x_2)$. Rescale the metric such that $d(x_1,
x_2) = 1$, let $y_2, z_2$ denote the blowups of $x_2$ in $Y[(1,3)]$, and
consider the modified Cayley--Menger matrix of $Y[(1,3)]$ with respect to
$(x_1, x_2, y_2, z_2)$:
\[
A = \begin{pmatrix} 2 & -2 & -2 \\ -2 & 2 & -2 \\ -2 & -2 & 2 \end{pmatrix}.
\]
As $\det A < 0$, it follows by Theorem~\ref{Tschoenberg} that $Y[(1,3)]$,
hence $X[\bn]$, is not Euclidean. (Note, these two uses of
Theorem~\ref{Tschoenberg} can also be avoided by using ``visual Euclidean
geometry''. For instance, in the latter case $x_2, y_2, z_2$ are the
vertices of an equilateral triangle with edge-length $2$, and drawing
three unit spheres centered at these vertices reveals there is no common
intersection point $x_1$.)

This shows the existence of the unique index $j \in [1,k]$ such that
$n_{x_j} = 2$ and (2)(a)\ all other $n_{x_i} = 1$. If $k=2$ then the
result is easy to show, so we now suppose that $k \geq 3$. Let $x_{j_0}$
denote any point in $X \setminus \{ x_j \}$ that is closest to $x_j$.
Since $X[\bn]$ is also Euclidean, considering the degenerate (Euclidean)
triangle with vertices $x_{j_0}, x_j$, and the blowup $x'_j$ of $x_j$
shows that these vertices are collinear, and in fact, $x_{j_0} = (x_j +
x'_j)/2$. In turn, this implies that a ``closest point'' $x_{j_0} \in X$
to $x_j$ (and hence the index $j_0$) is unique.

Next, denote by $l \geq 1$ the dimension of the span $V$ of $\{ x_i -
x_{j_0} : i \neq j \}$. Relabel the $x_i$ via: $y_0 := x_{j_0}, y_1 :=
x_j$, and choose any enumeration of the remaining points in $X$ as $y_2,
\dots, y_{k-1}$, such that $y_2 - y_0, \dots, y_{l+1} - y_0$ form a basis
of $V$. Since $X$ is Euclidean, a simple check shows that the modified
Cayley--Menger matrix of $X$ with respect to $(y_0, \dots, y_{k-1})$ is
precisely the Gram matrix
\[
(d(y_0, y_i)^2 + d(y_0, y_{i'})^2 - d(y_i, y_{i'})^2)_{i,i'=1}^{k-1} = (
2 \langle y_{i'} - y_0, y_i - y_0 \rangle )_{i,i'=1}^{k-1}.
\]
Now (2)(b) follows from the claim that if $y_1 - y_0$ is in the span of
$\{ y_i - y_0 : 2 \leq i \leq l+1 \}$, then this matrix uniquely
determines $y_1 \in \R^r$. To show the claim, write $y_1 - y_0 =
\sum_{i=2}^{l+1} c_i (y_i - y_0)$, and take inner products to obtain the
linear system
\[
\sum_{i=2}^{l+1} \langle y_{i'} - y_0, y_i - y_0 \rangle c_i = \langle
y_{i'} - y_0, y_1 - y_0 \rangle, \qquad 2 \leq i' \leq l+1.
\]
But the left-hand side is the product of the Gram matrix of $\{ y_{i'} -
y_0 : 2 \leq i' \leq l+1 \}$ against the vector $(c_2, \dots,
c_{l+1})^T$. As the vectors $y_{i'} - y_0$ are independent, their Gram
matrix is invertible, so this system determines all $c_i$ uniquely. In
particular, if $y_1 = x_j$ is in the affine hull of $X \setminus \{ x_j
\}$, then $X[\bn]$ cannot be Euclidean since $n_{x_j} > 1$. This proves
(2)(b).

Finally, since $k \geq 3$, we have $l \geq 1$. Now the definition of the
blowup $X[\bn]$ implies:
\[
\| y_i - y_1 \|_2^2 = \| y_i - (2 y_0 - y_1) \|_2^2, \qquad 2 \leq i \leq
l+1
\]
or in other words,
\[
\| (y_i - y_0) - (y_1 - y_0) \|_2^2 = \| (y_i - y_0) + (y_1 - y_0)
\|_2^2.
\]
Simplifying this equality yields that $y_1 - y_0$ is orthogonal to $y_2 -
y_0, \dots, y_{l+1} - y_0$. This implies $y_1 - y_0$ is orthogonal to all
of $V$ -- which was the assertion (2)(c).
\end{proof}

\subsection{Real-closed analogues}

For completeness, we now provide analogues of some of the results proved
above, over arbitrary real-closed fields. As this subsection is not used
anywhere else in the paper, we will be brief, and also assume familiarity
with real-closed fields; we refer the reader to the opening chapter
of~\cite{BCR} for this.

In this part, we suppose $\K$ is a real-closed field, where we denote the
non-negative elements in $\K$ by: $r \geq 0$. Also let $\overline{\K} =
\K[\sqrt{-1}]$ denote an algebraic closure of $\K$, where we fix a choice
of ``imaginary'' square root $i = \sqrt{-1}$ in $\overline{\K}$. Then
several ``real'' notions can be defined over $\K, \overline{\K}$:
\begin{enumerate}
\item A symmetric matrix $A \in \K^{k \times k}$ is said to be
\emph{positive semidefinite} (respectively, \emph{positive definite}) if
$x^T A x \geq 0$ (respectively, $x^T A x > 0$) for all nonzero vectors $x
\in \K^n$.

\item A matrix $A \in \K^{k \times k}$ is \emph{orthogonal} if $A A^T =
\Id_k$.

\item Given $z = x + iy \in \overline{\K}$ with $x,y \in \K$, we define
$\Re(z) := x$ and $\Im(z) := y$.

\item We say that a multivariable polynomial $p \in \overline{\K}[z_1,
\dots, z_k]$ is \emph{stable} if $p(z_1, \dots, z_k)$ $\neq 0$ whenever
$\Im(z_j) > 0\ \forall j$.
\end{enumerate}

Now (an analogue of) the spectral theorem holds for symmetric matrices $A
= A^T \in \K^{k \times k}$. Moreover, every such matrix $A$ is positive
semidefinite if and only if it has non-negative eigenvalues in $\K$ --
and if so, then it has a positive semidefinite square root $\sqrt{A}$.
One also has:

\begin{proposition}\label{Prealclosed}
Fix an integer $k \geq 1$.
\begin{enumerate}
\item Suppose $A_1, \dots, A_m, B \in \K^{k \times k}$ are symmetric
matrices, with all $A_j$ positive semidefinite. Then the polynomial
\[
p(z_1, \dots, z_m) := \det \left( B + \sum_{j=1}^m z_j A_j \right)
\]
is either stable or identically zero.

\item Suppose $p \in \K[z_1, \dots, z_m]$ is homogeneous of degree $k$.
If $p$ is stable, then $p$ is Lorentzian (equivalently,
strongly/completely log-concave, whenever $\log(\cdot)$ is defined over
$\K$).
\end{enumerate}
\end{proposition}

The two parts of this proposition were proved over $\K = \R$ in
\cite{BB1,BH}, respectively. Thus, they hold over an arbitrary real
closed field $\K$ by Tarski's principle, since all of these notions can
be expressed in the first-order language of ordered fields.
Also note that the equivalence in the second part indeed makes sense over
some real-closed fields, e.g.\ $\K = \R$, or $\K$ the field of convergent
(real) Puiseux series with rational powers, where it was proved in
\cite[Theorem 3.19]{BH}. The notions of Lorentzian and
strongly/completely log-concave polynomials over the latter field can be
found in \cite{BH} (see pp.~861).

\begin{remark}
Schoenberg's Euclidean embeddability result (Theorem~\ref{Tschoenberg})
-- stated and used above -- turns out to have an alternate
characterization in a specific real-closed field: the convergent
generalized Puiseux series $\R \{ t \}$ with real powers. Recall, the
elements of $\R \{ t \}$ are series
\[
\sum_{n=0}^\infty c_n t^{\alpha_n}, \qquad c_0, c_1, \ldots \in \R,
\]
satisfying:
(a)~the exponents $\alpha_0 < \alpha_1 < \cdots$ are real;
(b)~$\{ -\alpha_n : n \geq 0, \ c_n \neq 0 \}$ is well-ordered; and
(c)~$\sum_n c_n t^{\alpha_n}$ is convergent on a punctured open disk
around the origin. Such a series is defined to be positive if its leading
coefficient is positive. It is known (see e.g.\ \cite[Section
1.5]{Speyer}) that this field is real-closed; moreover, its algebraic
closure is the degree-$2$ extension $\mathbb{C} \{ t \}$, where 
real coefficients in the definition above are replaced by complex ones.
Now we claim that the following holds.
\end{remark}

\begin{theorem}
A finite metric space $X = \{ x_0, x_1, \dots, x_k \}$ isometrically
embeds inside Hilbert space $\ell^2$ (over $\R$) if and only if the
matrix $E_X := (t^{d(x_i, x_j)^2})_{i,j=0}^k$ is positive semidefinite in
$\R \{ t \}$.
\end{theorem}

\begin{proof}
This follows from a chain of equivalences:
$X$ is Euclidean-embeddable if and only if (by~\cite{Schoenberg38b},
via Theorem~\ref{Tschoenberg}) the matrix $(-d(x_i,x_j)^2)_{i,j=0}^k$ is
conditionally positive semidefinite,
if and only if (by~\cite{Schoenberg38b}) $(\exp (-\lambda
d(x_i,x_j)^2))_{i,j=0}^k$ is positive semidefinite for all $\lambda \geq
0$,
if and only if (replacing $e^{-\lambda} \leadsto q \in (0,1)$ and via
\cite[Section~1.5]{Speyer}) $E_X$ is positive semidefinite in $\R \{ t
\}$.
\end{proof}

We conclude this part by formulating the promised tropical version of
some of our results above. As the notion of a $\K_{\geq 0}$-valued metric
(can be easily defined, but) is not very common in the literature, we
formulate the next result in greater generality.

\begin{theorem}
Suppose $\K$ is a real-closed field, and $\Delta, M \in \K^{k \times k}$
are symmetric matrices, with $\Delta$ a positive definite diagonal
matrix.
\begin{enumerate}
\item The multi-affine polynomials
\[
p_\pm(z_1, \dots, z_k) := \det (\pm \Delta + \Delta_{\bf z} M), \qquad
{\bf z} \in \overline{\K}^k
\]
are stable, with coefficients in $\K$.

\item Define the homogenization $\widetilde{p}(z_0, z_1, \dots, z_k) :=
p(z_0 \Delta + \Delta_{\bf z} M)$. The following are equivalent:
\begin{enumerate}
\item $\widetilde{p}$ is stable.

\item $\widetilde{p}$ is Lorentzian (equivalently, strongly/completely
log-concave, whenever $\log(\cdot)$ is defined over $\K$).

\item All coefficients of $\widetilde{p}$ lie in $\K_{\geq 0}$.

\item $p(1,\dots,1) > 0$, and the polynomial $(z_1, \dots, z_k) \mapsto
\displaystyle \frac{p(z_1, \dots, z_k)}{p(1,\dots,1)}$ is stable and has
non-negative coefficients that sum to $1$.

\item The matrix $M$ is positive semidefinite.
\end{enumerate}
\end{enumerate}
\end{theorem}

One can prove this result from the same result over $\K = \R$, via
Tarski's principle. Alternately, the case of $\K = \R$ was essentially
proved in Theorem~\ref{Tlorentz}; the slightly more general versions here
(for $\K = \R$, with $M$ instead of $\D_G$) require only minimal
modifications to the earlier proofs. These proofs (for $\K = \R$) go
through over arbitrary real-closed $\K$ with minimal further
modifications, given Proposition~\ref{Prealclosed}.

We end the paper with some observations. The results in this work reveal
several novel invariants associated to all finite connected unweighted
graphs -- more generally, to all finite metric spaces $X$:
\begin{enumerate}
\item The (real-stable) blowup-polynomials $p_X(\bn)$ and $u_X(n)$.

\item The degrees of $p_X, u_X$ -- notice by Proposition~\ref{Ppath} that
the degrees can be strictly less than the number of points in $X$, even
when $X$ is not a blowup of a smaller graph.

\item The largest root of $u_X(n)$ (which is always positive). By
Corollary~\ref{Cdistancespectrum}, for $X = G$ a graph this equals
$\frac{2}{2+\lambda}$, where $\lambda$ is the smallest eigenvalue of
$D_G$ above $-2$.

\item The blowup delta-matroid $\mathcal{M}_{\D_X}$; and for $X$ an
unweighted tree, the delta-matroid $\mathcal{M}'(X)$ which is
combinatorial rather than matrix-theoretic.
\end{enumerate}

\noindent It would be interesting and desirable to explore if -- say for
$X = G$ a graph metric space -- these invariants can be related to
``traditional'' combinatorial graph invariants.

\appendix

\section{From the matrix $\mathcal{C}_X$ to the univariate
blowup-polynomial}

In this section, we show a few peripheral but related
results for completeness. As the proofs of Theorems~\ref{Tstable}
and~\ref{Tlorentz} reveal, the matrix $\mathcal{C}_G =
(-\Delta_{\ba_G})^{-1/2} \D_G (-\Delta_{\ba_G})^{-1/2}$ plays an
important role in determining the real-stability and Lorentzian nature of
$p_G(\cdot)$ and $\widetilde{p}_G(\cdot)$, respectively. We thus take a
closer look at $\mathcal{C}_G$ in its own right. Here, we will work over
an arbitrary finite metric space $X$.

\begin{proposition}\label{PCmatrix}
Given a finite metric space $X = \{ x_1, \dots, x_k \}$ with $k \geq 2$,
the (real) matrix $\mathcal{C}_X$ has at least two positive eigenvalues,
and this bound is sharp. More strongly, for any $k \geq 3$, there exists
a finite metric space $Y = \{ y_1, \dots, y_k \}$ for which
$\mathcal{C}_Y$ has exactly two positive eigenvalues.
\end{proposition}

\begin{proof}
Suppose $i \neq j$ are such that $d(x_i, x_j)$ is minimal among the
off-diagonal entries of the distance matrix $D_X$ (or $\D_X$).
Using~\eqref{Ecmatrix}, the corresponding principal $2 \times 2$
submatrix of $\mathcal{C}_X$ is $\begin{pmatrix} 1 & 1/2 \\ 1/2 & 1
\end{pmatrix}$. This is positive definite, hence has two positive
eigenvalues. But then so does $\mathcal{C}_X$, by the Cauchy interlacing
theorem.

We now show that this bound is tight, first for $k=3$. Consider a
Euclidean isosceles triangle with vertices $x_1, x_2, x_3$ and
edge-lengths $1,a,a$ for any fixed scalar $a>3$. Thus, $X = \{ x_1, x_2,
x_3 \}$; now a suitable relabeling of the $x_i$ gives:
\[
\D_X = \begin{pmatrix}
2 & 1 & a \\
1 & 2 & a \\
a & a & 2a
\end{pmatrix}.
\]
Now $\det \D_X = 2a(3-a) < 0$, so $\det \mathcal{C}_X = (3-a)/4 < 0$.
As $\mathcal{C}_X$ has two positive eigenvalues, the third eigenvalue
must be negative.

Finally, for $k > 3$, we continue to use $X = \{ x_1, x_2, x_3 \}$ as in
the preceding paragraph, and consider the blowup $Y := X[(1,1,k-2)]$. By
Lemma~\ref{Lwellbehaved}(3), $\D_Y$ is a block $3 \times 3$ matrix with
the $(3,3)$ block of size $(k-2) \times (k-2)$. Let $V \subset \R^k$
denote the subspace of vectors whose first two coordinates are zero and
the others add up to zero; then $V \subset \ker \D_Y$ and $\dim V = k-3$.
On the ortho-complement $V^\perp \subset \R^k$ (with basis $e_1, e_2$,
and $e_3 + \cdots + e_k$), $\D_Y$ acts as the matrix $\D_X \cdot \diag
(1, 1, k-2)$. Thus, the remaining three eigenvalues of $\D_Y$ are those
of $\D_X \cdot \diag (1, 1, k-2)$, or equivalently, of its conjugate
\[
C := \diag(1,1,\sqrt{k-2}) \cdot \D_X \cdot \diag(1,1,\sqrt{k-2}),
\]
which is real symmetric. Moreover, $\D_X$ has a negative eigenvalue,
hence is not positive semidefinite. But then the same holds for $C$,
hence for $\D_Y$, and hence for $\mathcal{C}_Y$. From above,
$\mathcal{C}_Y$ has at least two positive eigenvalues, and it has a
kernel of dimension at least $k-3$ because $\D_Y$ does.
\end{proof}

The next result starts from the calculation in Equation~\eqref{Emixed}.
The second determinant in the final expression is reminiscent of an
important tool used in proving the Kadison--Singer conjecture in
\cite{MSS2} (via Weaver's conjecture $KS_2$): the \textit{mixed
characteristic polynomial} of a set of positive semidefinite matrices
$A_j \in \R^{k \times k}$:
\begin{equation}\label{Emixeddefn}
\mu[A_1, \dots, A_m](z_0) := \left. \left( \prod_{j=1}^m (1 -
\partial_{z_j}) \right) \det \left( z_0 \Id_k + \sum_{j=1}^m z_j A_j
\right) \right|_{z_1 = \cdots = z_m = 0}.
\end{equation}

In the setting of Theorem~\ref{Tlorentz} where $\mathcal{C}_X$ is
positive semidefinite, the matrices $A_j$ in question --
see~\eqref{Emixed} -- are $\sqrt{\mathcal{C}_X} E_{jj}
\sqrt{\mathcal{C}_X}$, which are all positive semidefinite. It turns out
that their mixed characteristic polynomial is intimately connected to
both $\mathcal{C}_X$ and to $u_X(\cdot)$:

\begin{proposition}\label{Pmixed}
Suppose $(X,d)$ is a finite metric space such that $\D_X$ (equivalently,
$\mathcal{C}_X$) is positive semidefinite. Up to a real scalar, the mixed
characteristic polynomial of the matrices $\{
\sqrt{\mathcal{C}_X} E_{jj} \sqrt{\mathcal{C}_X} : 1 \leq j \leq k \}$
equals the characteristic polynomial of $\mathcal{C}_X$, as well as the
``inversion'' $z_0^k u_X(z_0^{-1})$ of the univariate blowup-polynomial.
\end{proposition}

\begin{proof}
We begin with the following claim: \textit{If $p(z_0, z_1, \dots, z_k)$
is a polynomial, then}
\begin{align}\label{Emap}
\begin{aligned}
&\ \left. \left( \prod_{j=1}^k (1 - \partial_{z_j}) \right) (p) (z_0,
z_1, \dots, z_k) \right|_{z_1 = \cdots = z_k = 0}\\
= &\ \left. MAP_{z_1, \dots, z_k}(p) (z_0, z_1, \dots, z_k) \right|_{z_1
= \cdots = z_k = -1}.
\end{aligned}
\end{align}

Here the operator $MAP_{z_1, \dots, z_k}$ kills all higher-order terms in
$z_1, \dots, z_k$, and retains only the ``multi-affine'' monomials
$\prod_{j=0}^k z_j^{m_j}$ with all $m_1, \dots, m_k \leq 1$ (but
arbitrary $m_0$).

To show the claim, write
\[
p(z_0, z_1, \dots, z_k) = \sum_{{\bf m} \in \mathbb{Z}_{\geq 0}^k}
a_{\bf m}(z_0) \prod_{j=1}^k z_j^{m_j},
\]
where each $a_{\bf m}(z_0)$ is a polynomial. Now note that applying a
differential ``monomial operator'' and then specializing at the origin
kills all but one monomials:
\[
\left. \prod_{i \in I} \partial_{z_i} \left( \prod_{j=1}^k z_j^{m_j} \right)
\right|_{z_1 = \cdots = z_k = 0} = {\bf 1}(m_i = 1\ \forall i \in I)
\cdot {\bf 1}(m_j = 0\ \forall j \not\in I), \
\forall I \subset \{ 1, \dots, k \}.
\]

Thus, applying $\prod_{i=1}^k (1 - \partial_{z_i})$ and specializing at
the origin will kill all but the multi-affine (in $z_1, \dots, z_k$)
monomials in $p$. Consequently, we may replace $p$ on the left-hand side
in~\eqref{Emap} by
\[
q(z_0, \dots, z_k) := MAP_{z_1, \dots, z_k}(p)(z_0, z_1, \dots, z_k) =
\sum_{J \subset \{ 1, \dots, k \}} a_J(z_0) \prod_{j \in J} z_j,
\]
say, where each $a_J(z_0)$ is a polynomial as above. Now compute:
\begin{align*}
&\ \left. \left( \prod_{j=1}^k (1 - \partial_{z_j}) \right) q(z_0, z_1,
\dots, z_k) \right|_{z_1 = \cdots = z_k = 0}\\
= &\ \left. \sum_{I \subset \{ 1, \dots, k \}} (-1)^{|I|} \prod_{i \in I}
\partial_{z_i} \cdot \sum_{J \subset \{ 1, \dots, k \}} a_J(z_0) \prod_{j
\in J} z_j \right|_{z_1 = \cdots = z_k = 0}\\
= &\ \sum_{J \subset \{ 1, \dots, k \}} (-1)^{|J|} a_J(z_0)
= q(z_0, -1, -1, \dots, -1).
\end{align*}

This shows the claim. Using this, we and assuming that $\D_X$ is positive
semidefinite, we show the first assertion -- that the mixed
characteristic polynomial of the matrices $\sqrt{\mathcal{C}_X} E_{jj}
\sqrt{\mathcal{C}_X}$ is the characteristic polynomial of
$\mathcal{C}_X$:
\begin{align*}
\mu[\{ \sqrt{\mathcal{C}_X} E_{jj} \sqrt{\mathcal{C}_X} : 1 \leq j \leq k
\}](z_0) = &\ \det \left( z_0 \Id_k + \sum_{j=1}^k (-1)
\sqrt{\mathcal{C}_X} E_{jj} \sqrt{\mathcal{C}_X} \right)\\
= &\ \det(z_0 \Id_k - \mathcal{C}_X).
\end{align*}

This shows one assertion. Next, using~\eqref{Ecomput} specialized at $z_1
= \cdots = z_k = -1$, this expression also equals (up to a scalar) the
inversion of $u_X$, as asserted:
\begin{align*}
= &\ \det(-\Delta_{\ba_X})^{-1} \det(z_0 \cdot (-\Delta_{\ba_X}) - \D_X)
= \det(\Delta_{\ba_X})^{-1} z_0^k \det(\Delta_{\ba_X} + z_0^{-1} \D_X)\\
= &\ \det(\Delta_{\ba_X})^{-1} z_0^k u_X(z_0^{-1}). \qedhere
\end{align*}
\end{proof}

We conclude this section by briefly discussing another invariant of the
metric space $(X,d)$, one that is related to the matrix $\mathcal{C}_X$
and the polynomial $u_X(\cdot)$.

\begin{definition}
Given a finite metric space $(X,d)$, define $\alpha_{\max}(X) \in \R$ to
be the largest root, or ``maxroot'', of (the real-rooted polynomial)
$u_X(\cdot)$.
\end{definition}

Notice that $u_X(x+\alpha_{\max}(X))$ has only non-positive roots, and
(say after dividing by its leading coefficient,) only non-negative
coefficients. In particular, these coefficients yield a (strongly)
log-concave sequence for every finite metric space -- e.g., for every
finite connected unweighted graph.

\begin{proposition}\label{Pmaxroot}
Given a finite metric space $X = \{ x_1, \dots, x_k \}$ with $k \geq 2$,
the root $\alpha_{\max}(X)$ is positive. In particular, $u_X(\cdot)$ has
monomials with both positive and negative coefficients.

Furthermore, given any integer $n \geq 1$, we have
\[
\alpha_{\max}(X[n(1,1,\dots,1)]) = \frac{\alpha_{\max}(X)}{n}.
\]
\end{proposition}

For example, for $X$ the $k$-point discrete metric space (so $\D_X = {\bf
1}_{k \times k} + \Id_k$ is positive definite) -- or equivalently, the
complete graph $K_k$ -- we have $\alpha_{\max}(X) = 2$, by the
computation in Equation~\eqref{Ecomplete}.

\begin{proof}
Recall that $u_X(x) = \det(\Delta_{\ba_X} + x \D_X)$, and so $u_X(0) =
\prod_{i=1}^k a_i \neq 0$. Now if $\alpha$ is any root of $u_X$ (hence
real and nonzero), we obtain similarly to the preceding proofs in this
section:
\begin{equation}\label{Ecomputation}
0 = u_X(\alpha) = \alpha^k \det (-\Delta_{\ba_X}) \cdot
\det(\mathcal{C}_X - \alpha^{-1} \Id_k),
\end{equation}
where $\mathcal{C}_X = (-\Delta_{\ba_X})^{-1/2} \D_X
(-\Delta_{\ba_X})^{-1/2}$ as above. But since $\mathcal{C}_X$ has at
least two positive eigenvalues (by Proposition~\ref{PCmatrix}), $u_X$ has
at least two positive roots, and hence $\alpha_{\max}(X) > 0$.

The next statement follows directly by Theorem~\ref{Tmetricmatrix}, which
implies that the constant and linear terms in $u_X$ have opposite signs.
More information is obtained by using Descartes' rule of signs (see
e.g.~\cite{Descartes}), which implies that there are at least two sign
changes in the coefficients of $u_X$.

It remains to compute $\alpha_{\max}(X[n(1,1,\dots,1)])$. Using
Lemma~\ref{Lwellbehaved}(3),
\[
\D_{X[n(1,1,\dots,1)]} = \D_X \otimes {\bf 1}_{n \times n}, \qquad
\mathcal{C}_{X[n(1,1,\dots,1)]} = \mathcal{C}_X \otimes {\bf 1}_{n \times
n},
\]
the Kronecker products, under a suitable relabeling of indices. 
In particular, the eigenvalues of the latter Kronecker product are the
Minkowski product of the spectra $\sigma(\mathcal{C}_X)$ and $\sigma({\bf
1}_{n \times n}) = \{ n, 0, 0, \dots, 0 \}$. We now make a series of
reductions:

$\alpha \in \R$ is a positive root of $u_{X[n(1,\dots,1)]}$,

if and only if $\alpha^{-1}$ is a positive eigenvalue of
$\mathcal{C}_{X[n (1,1,\dots,1)]}$ (using~\eqref{Ecomputation}),

if and only if (from above) $\alpha^{-1}$ equals $n$ times a positive
eigenvalue of $\mathcal{C}_X$,

if and only if $\frac{1}{n \alpha}$ is a positive eigenvalue of
$\mathcal{C}_X$,

if and only if $n \alpha$ is a positive root of $u_X$ (again
using~\eqref{Ecomputation}).

\noindent Since $\alpha_{\max}(X) > 0$, this proves the result.
\end{proof}

\section{Monoid morphism computations towards Theorem~\ref{Tmonoid}}

Here we show parts~(1) and~(3) of Theorem~\ref{Tmonoid}, for
completeness. For~(1), begin by using Lemma~\ref{Lmonoid} to check that
$((1,\dots,1)^T, 0_{k \times k})$ is indeed the identity element of
$\mathcal{M}_{\bf n}(R)$; clearly, it is sent under the given map $\Psi :
(\ba, D) \mapsto M(\ba,D)$ to $\Id_K$. Next, given $(\ba,D),(\ba',D') \in
\mathcal{M}_{\bf n}(R)$, and $1 \leq i \leq k$, first compare the $(i,i)$
blocks of their $\Psi$-images in part~(1):
\[
M((\ba,D) \circ (\ba',D'))_{ii} = a_i a'_i \Id_{n_i} +
\left( a_i d'_{ii} + d_{ii} a'_i + \sum_{l=1}^k d_{il} \cdot {\bf q}_l^T
{\bf p}_l \cdot d'_{li} \right) {\bf p}_i {\bf q}_i^T,
\]
whereas on the other side,
\begin{align*}
&\ (M(\ba,D) M(\ba',D'))_{ii}\\
= &\ (a_i \Id_{n_i} + d_{ii} {\bf p}_i {\bf q}_i^T)
(a'_i \Id_{n_i} + d'_{ii} {\bf p}_i {\bf q}_i^T) +
\sum_{l \neq i} d_{il} {\bf p}_i {\bf q}_l^T \cdot {\bf p}_l {\bf q}_i^T
\cdot d'_{li}\\
= &\ (a_i \Id_{n_i} + d_{ii} {\bf p}_i {\bf q}_i^T)
(a'_i \Id_{n_i} + d'_{ii} {\bf p}_i {\bf q}_i^T) +
{\bf p}_i {\bf q}_i^T \cdot \sum_{l \neq i} d_{il} \cdot {\bf q}_l^T
{\bf p}_l \cdot d'_{li}.
\end{align*}
It is easily verified that these two quantities are equal.

For the off-diagonal blocks, fix $1 \leq i \neq j \leq k$ and compute:
\[
M((\ba,D) \circ (\ba',D'))_{ij} = \left( a_i d'_{ij} + d_{ij} a'_j +
\sum_{l=1}^k d_{il} \cdot {\bf q}_l^T {\bf p}_l \cdot d'_{lj} \right)
{\bf p}_i {\bf q}_j^T,
\]
whereas on the other side,
\begin{align*}
&\ (M(\ba,D) M(\ba',D'))_{ij}\\
= &\ \left( a_{ii} \Id_{n_i} \cdot d'_{ij} {\bf p}_i {\bf q}_j^T +
d_{ii} {\bf p}_i {\bf q}_i^T \cdot d'_{ij} {\bf p}_i {\bf q}_j^T
\right)\\
&\ + \left( d_{ij} {\bf p}_i {\bf q}_j^T \cdot a'_{jj} \Id_{n_j} +
d_{ij} {\bf p}_i {\bf q}_j^T \cdot d'_{jj} {\bf p}_j {\bf q}_j^T
\right)
+ {\bf p}_i {\bf q}_j^T \cdot \sum_{l \neq i,j} d_{il} d'_{li} {\bf
q}_l^T {\bf p}_l.
\end{align*}
Once again, it is easy to check that both expressions are equal.

We now show Theorem~\ref{Tmonoid}(3); denote the right-hand side in it by
\[
X := M((a_1^{-1}, \dots, a_n^{-1})^T, - \Delta_\ba^{-1} D N(\ba,
D)^{-1}),
\]
noting that all $a_i \in R^\times$. To show~(3), using~(1) it suffices to
show that the composition on the other side yields the identity, i.e.,
\[
((a_1^{-1}, \dots, a_n^{-1})^T, -\Delta_\ba^{-1} D N(\ba,
D)^{-1}) \circ (\ba, D) = ((1,\dots,1)^T, 0_{k \times k}).
\]
The equality of the first components is obvious; now check using
Lemma~\ref{Lmonoid} that the second component on the left-hand side
equals
\[
\Delta_\ba^{-1} \cdot D + (-\Delta_\ba^{-1} D N(\ba, D)^{-1})
N(\ba, D) = 0_{k \times k},
\]
which concludes the proof. \qed

\subsection*{Acknowledgments}
The authors would like to thank Alexander Belton, Petter Br\"and\'en,
Carolyn Chun, Karola Meszaros, Aaradhya Pandey, Alan Sokal, and Cynthia
Vinzant for useful discussions. We also thank the anonymous referee for
carefully reading the manuscript and providing useful feedback.
P.N.~Choudhury was partially supported by
INSPIRE Faculty Fellowship research grant\hfill\break
DST/INSPIRE/04/2021/002620 (DST, Govt.\ of India),
IIT Gandhinagar Internal Project grant IP/IITGN/MATH/PNC/2223/25,
C.V.\ Raman Postdoctoral Fellowship 80008664 (IISc), and
National Post-Doctoral Fellowship (NPDF) PDF/2019/000275 from SERB
(Govt.\ of India).
A.~Khare was partially supported by
Ramanujan Fellowship grant SB/S2/RJN-121/2017,
MATRICS grant MTR/2017/000295, and
SwarnaJayanti Fellowship grants SB/SJF/2019-20/14 and DST/SJF/MS/2019/3
from SERB and DST (Govt.~of India),
by a Shanti Swarup Bhatnagar Award from CSIR (Govt.\ of India),
by grant F.510/25/CAS-II/2018(SAP-I) from UGC (Govt.~of India), and
by the DST FIST program 2021 [TPN--700661].



\begin{thebibliography}{88}
\bibitem{AGV3}
N.~Anari, S.~Oveis Gharan, and C.~Vinzant.
\newblock Log-concave polynomials, entropy, and a deterministic
approximation algorithm for counting bases of matroids.
\newblock \href{http://dx.doi.org/10.1109/FOCS.2018.00013}{\em FOCS 2018
Proc.} (59th Annual IEEE Symposium), pp.\ 35--46, 2018.

\bibitem{AGV}
N.~Anari, S.~Oveis Gharan, and C.~Vinzant.
\newblock Log-concave polynomials, I: Entropy and a deterministic
approximation algorithm for counting bases of matroids.
\newblock \href{http://dx.doi.org/10.1215/00127094-2020-0091}{\em Duke
Math.\ J.}, 170(16):3459--3504, 2021.

\bibitem{AH}
M.~Aouchiche and P.~Hansen.
\newblock Distance spectra of graphs: a survey.
\newblock \href{http://dx.doi.org/10.1016/j.laa.2014.06.010}{\em 
Linear Algebra Appl.}, 458:301--386, 2014.

\bibitem{BCR}
J.~Bochnak, M.~Coste, and M.-F.~Roy.
\newblock {\em G\'eom\'etrie alg\'ebrique r\'eelle}.
\newblock In: \href{https://www.springer.com/gp/book/9783540169512}{\em
Ergebnisse der Mathematik und ihrer Grenzgebiete}, Vol.\ 12, x+376 pp.,
Springer--Verlag, Berlin, Heidelberg, 1987. 

\bibitem{BB1}
J.~Borcea and P.~Br\"and\'en.
\newblock Applications of stable polynomials to mixed determinants:
Johnson's conjectures, unimodality, and symmetrized Fischer products.
\newblock \href{http://dx.doi.org/10.1215/00127094-2008-018}{\em Duke
Math.\ J.}, 143(2):205--223, 2008.

\bibitem{BB2}
J.~Borcea and P.~Br\"and\'en.
\newblock P\'olya--Schur master theorems for circular domains and their
boundaries.
\newblock \href{http://dx.doi.org/10.4007/annals.2009.170.465}{\em Ann.\
of Math.\ (2)}, 170(1):465--492, 2009.

\bibitem{BB3}
J.~Borcea and P.~Br\"and\'en.
\newblock The Lee--Yang and P\'olya--Schur programs.
I. Linear operators preserving stability.
\newblock \href{http://dx.doi.org/10.1007/s00222-009-0189-3}{\em Invent.\
Math.}, 177(3):541--569, 2009.

\bibitem{BBL}
J.~Borcea, P.~Br\"and\'en, and T.M.~Liggett.
\newblock Negative dependence and the geometry of polynomials.
\newblock \href{http://dx.doi.org/10.1090/S0894-0347-08-00618-8}{\em J.\
Amer.\ Math.\ Soc.}, 22(2):521--567, 2009.

\bibitem{Bouchet1}
A.~Bouchet.
\newblock Greedy algorithms and symmetric matroids.
\newblock {\em Math Progr.}, 38:147--159, 1987.

\bibitem{Bouchet2}
A.~Bouchet.
\newblock Representability of $\Delta$-matroids.
\newblock In: {\em Proc.\ 6th Hungarian Coll.\ Combin.\ 1987}, Colloq.\
Math.\ Soc.\ J\'anos Bolyai, 52:167--182, 1988.

\bibitem{Branden}
P.~Br\"and\'en.
\newblock Polynomials with the half-plane property and matroid theory.
\newblock \href{http://dx.doi.org/10.1016/j.aim.2007.05.011}{\em Adv.\ in
Math.}, 216(1):302--320, 2007.

\bibitem{BH}
P.~Br\"and\'en and J.~Huh.
\newblock Lorentzian polynomials.
\newblock \href{http://dx.doi.org/10.4007/annals.2020.192.3.4}{\em Ann.\
of Math.\ (2)}, 192(3):821--891, 2020.

\bibitem{Cayley}
A.~Cayley.
\newblock On a theorem in the geometry of position.
\newblock {\em Cambridge Math.\ J.}, II:267--271, 1841.

\bibitem{metric}
J.~Cheeger, B.~Kleiner, and A.~Schioppa.
\newblock Infinitesimal structure of differentiability spaces, and metric
differentiation.
\newblock \href{http://dx.doi.org/10.1515/agms-2016-0005}{\em Anal.\
Geom.\ Metric Spaces}, 4(1):104--159, 2016.

\bibitem{COSW}
Y.-B.~Choe, J.G.~Oxley, A.D.~Sokal, and D.G.~Wagner.
\newblock Homogeneous multivariate polynomials with the half-plane
property.
\newblock \href{http://dx.doi.org/10.1016/S0196-8858(03)00078-2}{\em
Adv.\ in Appl.\ Math.}, 32(1--2):88--187, 2004.

\bibitem{CK-tree}
P.N.~Choudhury and A.~Khare.
\newblock Distance matrices of a tree: two more invariants, and in a
unified framework.
\newblock
\href{http://dx.doi.org/10.1016/j.ejc.2023.103787}{\em Eur.\ J.\
Combin.}, 115, article \#103787 (30 pp.), 2024.

\bibitem{Descartes}
R.~Descartes.
\newblock Le G\'eom\'etrie.
\newblock
\href{https://archive.org/details/geometryofrene00desc}{Appendix} to
\textit{Discours de la m\'ethode}, 1637.

\bibitem{DL}
S.~Drury and H.~Lin.
\newblock Some graphs determined by their distance spectrum.
\newblock \href{http://dx.doi.org/10.13001/1081-3810.3613}{\em Electr.\
J.\ Linear Alg.}, 34:320--330, 2018.

\bibitem{Frechet0}
M.R.~Fr\'echet.
\newblock Les dimensions d'un ensemble abstrait.
\newblock \href{http://dx.doi.org/10.1007/BF01474158}{\em Math.\ Ann.},
68:145--168, 1910.

\bibitem{GL}
R.L.~Graham and L.~Lov\'asz.
\newblock Distance matrix polynomials of trees.
\newblock \href{http://dx./doi.org/10.1016/0001-8708(78)90005-1}{\em
Adv.\ in Math.}, 29(1):60--88, 1978.

\bibitem{Graham-Pollak}
R.L.~Graham and H.O.~Pollak.
\newblock On the addressing problem for loop switching.
\newblock {\em Bell System Tech.\ J.}, 50:2495--2519, 1971.

\bibitem{GKR-critG}
D.~Guillot, A.~Khare, and B.~Rajaratnam.
\newblock Critical exponents of graphs.
\newblock \href{http://dx.doi.org/10.1016/j.jcta.2015.11.003}%
{\em J.\ Combin.\ Th.\ Ser.\ A}, 139:30--58, 2016.

\bibitem{Gurvits}
L.~Gurvits.
\newblock On multivariate Newton-like inequalities.
\newblock \href{http://dx.doi.org/10.1007/978-3-642-03562-3_4}{\em Adv.\
in Combin.\ Math.}, Springer, Berlin, pp.\ 61--78, 2009.

\bibitem{HHN}
H.~Hatami, J.~Hirst, and S.~Norine.
\newblock The inducibility of blow-up graphs.
\newblock \href{http://dx.doi.org/10.1016/j.jctb.2014.06.005}{\em J.\
Combin.\ Th.\ Ser.\ B}, 109:196--212, 2014.

\bibitem{KKOT}
J.~Kim, D.~K\"{u}hn, D.~Osthus, and M.~Tyomkyn.
\newblock A blow-up lemma for approximate decompositions.
\newblock \href{http://dx.doi.org/10.1090/tran/7411}{\em Trans.\ Amer.\
Math.\ Soc.}, 371(7):4655--4742, 2019.

\bibitem{KSS}
J.~Koml\'os, G.N.~S\'ark\"ozy, and E.~Szemer\'edi.
\newblock Blow-up lemma.
\newblock \href{http://dx.doi.org/10.1007/BF01196135}{\em Combinatorica},
17:109--123, 1997.

\bibitem{Laguerre1}
E.~Laguerre.
\newblock Sur les fonctions du genre z\'ero et du genre un.
\newblock {\em C.\ R.\ Acad.\ Sci.\ Paris}, 95:828--831, 1882.

\bibitem{LHWS}
H.~Lin, Y.~Hong, J.~Wang, and J.~Shu.
\newblock On the distance spectrum of graphs.
\newblock \href{http://http://dx.doi.org/10.1016/j.laa.2013.04.019}{\em
Linear Algebra Appl.}, 439(6):1662--1669, 2013.

\bibitem{Liu}
H.~Liu.
\newblock Extremal graphs for blow-ups of cycles and trees.
\newblock \href{http://dx.doi.org/10.37236/2856}{\em Electr.\ J.\
Combin.}, 20(1): art.\ \# P65, 16 pp., 2013.

\bibitem{MSS1}
A.W.~Marcus, D.A.~Spielman, and N.~Srivastava.
\newblock Interlacing families I:
Bipartite Ramanujan graphs of all degrees.
\newblock \href{http://dx.doi.org/10.4007/annals.2015.182.1.7}{\em Ann.\
of Math.\ (2)}, 182(1):307--325, 2015.

\bibitem{MSS2}
A.W.~Marcus, D.A.~Spielman, and N.~Srivastava.
\newblock Interlacing families II.
Mixed characteristic polynomials and the Kadison-Singer problem.
\newblock \href{http://dx.doi.org/10.4007/annals.2015.182.1.8}{\em Ann.\
of Math.\ (2)}, 182(1):327--350, 2015.

\bibitem{Menger0}
K.~Menger.
\newblock Untersuchungen \"uber allgemeine Metrik.
\newblock {\em Math.\ Ann.}:
\href{http://dx.doi.org/10.1007/BF01448840}{\em Part I}:
Vol.~100:75--163, 1928;
\href{http://dx.doi.org/10.1007/BF01455705}{\em Part II}:
Vol.~103:466--501, 1930.

\bibitem{Menger}
K.~Menger.
\newblock New foundation of Euclidean geometry.
\newblock \href{http://dx.doi.org/10.2307/2371222}{\em Amer.\ J.\ Math.},
53(4):721--745, 1931.

\bibitem{Polya1913}
G.~P\'olya.
\newblock \"Uber Ann\"aherung durch Polynome mit lauter reellen Wurzeln.
\newblock \href{http://dx.doi.org/10.1007/BF03016033}{\em Rend.\ Circ.\
Mat.\ Palermo}, 36:279--295, 1913.

\bibitem{Polya-Schur}
G.~P\'olya and I.~Schur.
\newblock \"Uber zwei Arten von Faktorenfolgen in der Theorie der
algebraischen Gleichungen.
\newblock \href{http://dx.doi.org/10.1515/crll.1914.144.89}%
{\em J.\ reine angew.\ Math.}, 144:89--113, 1914.

\bibitem{Royle-Sokal}
G.~Royle and A.D.~Sokal.
\newblock The Brown--Colbourn conjecture on zeros of reliability
polynomials is false.
\newblock \href{http://dx.doi.org/10.1016/j.jctb.2004.03.008}{\em J.\
Combin.\ Th.\ Ser.\ B}, 91(2):345--360, 2004.

\bibitem{Sokal1}
A.D.~Sokal.
\newblock Bounds on the complex zeros of (di)chromatic polynomials and
Potts-model partition functions.
\newblock \href{http://dx.10.1017/S0963548300004612}{\em Combin.\
Probab.\ Comput.}, 10(1):41--77, 2001.

\bibitem{Sokal2}
A.D.~Sokal.
\newblock The multivariate Tutte polynomial (alias Potts model)
for graphs and matroids.
\newblock In: \href{http://dx.doi.org/10.1017/CBO9780511734885.009}{\em
Surveys in Combinatorics 2005}, London Math.\ Soc.\ Lect.\ Notes,
327:173--226 (chap.~8), 2005.

\bibitem{Schoenberg35}
I.J.~Schoenberg.
\newblock Remarks to Maurice Fr\'echet's article ``Sur la d\'efinition
  axiomatique d'une classe d'espace distanci\'es vectoriellement
  applicable sur l'espace de Hilbert''.
\newblock \href{http://dx.doi.org/10.2307/1968654}{\em Ann.\ of Math.\
(2)}, 36(3):724--732, 1935.

\bibitem{Schoenberg38b}
I.J.~Schoenberg.
\newblock Metric spaces and positive definite functions.
\newblock
\href{http://dx.doi.org/10.1090/S0002-9947-1938-1501980-0}{\em Trans.
  Amer. Math. Soc.}, 44(3):522--536, 1938.

\bibitem{Speyer}
D.E.~Speyer.
\newblock Horn's problem, Vinnikov curves, and the hive cone.
\newblock \href{http://dx.doi.org/10.1215/S0012-7094-04-12731-0}{\em Duke
Math.\ J.}, 127(3):395--427, 2005.

\bibitem{Wagner}
D.G.~Wagner.
\newblock Zeros of reliability polynomials and $f$-vectors of matroids.
\newblock \href{http://dx.doi.org/10.1017/S0963548399004162}{\em Combin.\
Probab.\ Comput.}, 9(2):167--190, 2000.

\bibitem{WW}
D.G.~Wagner and Y.~Wei.
\newblock A criterion for the half-plane property.
\newblock \href{http://dx.doi.org/10.1016/j.disc.2008.02.005}{\em
Discrete Math.}, 309(6):1385--1390, 2009.
\end{thebibliography}


\end{document}